\documentclass[11pt, reqno]{amsart}
\makeatletter
\usepackage{fullpage}
\usepackage{textcmds}  
\usepackage{amsmath, amssymb, amsfonts, amstext, verbatim, amsthm, mathrsfs}
\usepackage[mathcal]{eucal}
\usepackage{microtype}
\usepackage{graphicx}
\usepackage[all]{xy}
\usepackage{stmaryrd}

\usepackage{cite}
\usepackage[modulo]{lineno}
\usepackage[usenames]{color}
\usepackage{aliascnt} 
\usepackage[colorlinks=true,linkcolor=blue,citecolor=blue,urlcolor=blue,citebordercolor={0 0 1},urlbordercolor={0 0 1},linkbordercolor={0 0 1}]{hyperref} 
\usepackage[shortalphabetic]{amsrefs} 

\usepackage{enumitem}
\usepackage{xspace}
\usepackage{tikz}

\theoremstyle{plain}

\newcommand{\refnewtheoremn}[4]{%
\newaliascnt{#1}{#2}
\newtheorem{#1}[#1]{#3}
\aliascntresetthe{#1}
\expandafter\providecommand\csname #1autorefname\endcsname{#4}}

\newcommand{\refnewtheorem}[3]{\refnewtheoremn{#1}{#2}{#3}{#3}}

\def\makeCal#1{%
\expandafter\newcommand\csname c#1\endcsname{\mathcal{#1}}}
\def\makeBB#1{%
\expandafter\newcommand\csname b#1\endcsname{\mathbb{#1}}}
\def\makeFrak#1{%
\expandafter\newcommand\csname f#1\endcsname{\mathfrak{#1}}}
\def\makeScr#1{%
\expandafter\newcommand\csname s#1\endcsname{\mathscr{#1}}}

\count@=0
\loop
\advance\count@ 1
\edef\y{\@Alph\count@}%
\expandafter\makeCal\y
\expandafter\makeBB\y
\expandafter\makeFrak\y
\expandafter\makeScr\y
\ifnum\count@<26
\repeat

\newtheorem{thm}{Theorem}[section]

\refnewtheorem{prethm}{thm}{Pre-theorem}

\refnewtheorem{lem}{thm}{Lemma}
\refnewtheorem{claim}{thm}{Claim}
\refnewtheorem{cor}{thm}{Corollary}
\refnewtheorem{conj}{thm}{Conjecture}
\refnewtheorem{prop}{thm}{Proposition}
\refnewtheorem{quest}{thm}{Question}
\refnewtheorem{amplif}{thm}{Amplification}
\refnewtheorem{scholium}{thm}{Scholium}

\refnewtheorem{tablethm}{thm}{Table}
\refnewtheorem{assumption}{thm}{Assumption}
\refnewtheorem{conclusion}{thm}{Conclusion}
\refnewtheorem{problem}{thm}{Problem}

\theoremstyle{definition}
\refnewtheorem{rem}{thm}{Remark}
\refnewtheorem{defn}{thm}{Definition}
\refnewtheorem{notn}{thm}{Notation}
\refnewtheorem{const}{thm}{Construction}
\refnewtheorem{ex}{thm}{Example}
\refnewtheorem{fact}{thm}{Fact}

\newcommand{\op}[1]{\!\!\mathop{\rm ~#1}\nolimits}

\newcommand{\dual}{\vee}
\newcommand{\bung}{\cM}

\newcommand{\Sym}{\op{Sym}}

\newcommand{\iSpec}{\underline{\op{Spec}}}
\newcommand{\Perf}{\op{Perf}}
\newcommand{\Pic}{\op{Pic}}

\newcommand{\Cone}{\op{Cone}}
\renewcommand{\ss}{{\rm{ss}}}
\newcommand{\rank}{\op{rank}}
\newcommand{\QCoh}{\op{QCoh}}
\newcommand{\GL}{\op{GL}}
\newcommand{\Higgs}{\cH}
\newcommand{\Coh}{\op{Coh}}

\newcommand{\fg}{\mathfrak{g}}
\newcommand{\fz}{\mathfrak{z}}
\newcommand{\ft}{\mathfrak{t}}
\newcommand{\tw}[1]{(#1)}
\newcommand{\ttype}{\gamma}

\newcommand{\Hom}{\op{Hom}}
\newcommand{\iMap}{\cM ap}

\newcommand{\thickslash}{\mathbin{\!\!\pmb{\fatslash}}}
\newcommand{\hWt}{\op{highestWt}}

\newcommand{\frg}{\mathfrak{g}}
\newcommand{\bfX}{\mathbf{X}}

\begin{document}

\title{The equivariant Verlinde formula on the moduli of Higgs bundles}
\author[D. Halpern-Leistner]
{Daniel Halpern-Leistner \\ {\tiny \textit{with an Appendix by Constantin Teleman}} \\}

\begin{abstract}
We prove an analog of the Verlinde formula on the moduli space of semistable meromorphic $G$-Higgs bundles over a smooth curve for a reductive group $G$ whose fundamental group is free. The formula expresses the graded dimension of the space of sections of a positive line bundle as a finite sum whose terms are indexed by formal solutions of a generalized Bethe ansatz equation on the maximal torus of $G$.
\end{abstract}

\maketitle


Consider a smooth algebraic curve $\Sigma$ over $\bC$, a semisimple group $G$,\footnote{In the body of the paper, $G$ will denote an arbitrary reductive group such that $\pi_1(G)$ is free.} and the moduli space $M_G$ of semistable principal $G$-bundles over $\Sigma$. The Verlinde formula, conjectured in \cite{verlinde1988fusion} and proven in \cite{faltings1994proof,beauville1994conformal,kumar1994infinite} expresses the dimension of the space of ``generalized $\Theta$ functions,'' sections of powers of a determinant line bundle on $M_G$. This formula was studied intensely due to its relationship with conformal field theory (see for instance \cite{beauville1996conformal}), and the counter-intuitive expression for the dimension which emerges -- a priori it is not even obviously an integer.

We prove a version of this formula for the moduli space $H_G$ of semistable Higgs bundles, although the most natural statements are in terms of the stack of semistable Higgs bundles $\Higgs^{\ss}_G$. Because $\Higgs^\ss_G$ and $H_G$ are not compact, the space of sections of our line bundles will be infinite dimensional. In order to get a meaningful answer we use the action of $\bG_m$ on $\Higgs^\ss_G$ which scales the Higgs field. The line bundles $\cL$ we study will be equivariant with respect to this $\bG_m$ action, so their global sections $H^0(\Higgs_{G}^\ss,\cL)$ will have a canonical $\bG_m$-action. Our final formula expresses the \emph{graded dimension}
\[
\dim_{\bC^\ast} H^0(\Higgs_{G}^\ss,\cL) := \sum t^n \dim ( H^0(\Higgs_{G}^\ss,\cL)_{\text{weight }n})
\]

For simplicity, let us consider the case where $G = SL_2$. We denote the moduli stack of Higgs bundles $\Higgs := T^\ast \cM_G$, where $\cM_G$ denotes the stack of principal $G$-bundles, and let $\Higgs^\ss \subset \Higgs$ denote the open substack of semistable Higgs bundles. Let $\cL = \cO(h)$, where $\cO(1)$ is the positive generator of $\Pic(\cM_G)$ restricted to $\Higgs$. Our main result in this case states that $H^i(\cH^{\ss},\cL) = 0$ for $i>0$ and $\dim_{\bC^\ast} H^0(\cH^{\ss},\cL)$ is given by
\begin{equation} \label{eqn:ex_SL2}
\left( \frac{h+2}{2(1-t)}\right)^{g-1} \sum_{k=1}^{l+1}  (\sin \phi_k)^{2-2g} \left[ \frac{1}{(1-t)^2 + 4t\sin ^2 \phi_k} \cdot \left(1+ \frac{4t}{h+2} \frac{(1-t) - 2 \sin^2 \phi_k}{(1-t)^2 + 4t \sin^2 \phi_k} \right) \right]^{g-1},
\end{equation}
where $\phi_k = \frac{\pi k}{h+2} + t \phi_k^{(1)} + t^2 \phi_k^{(2)} + \cdots$ is the unique formal power series satisfying
\[
e^{-i 2(h+2) \phi_k} \left( \frac{1-t e^{2 i \phi_k}}{1-te^{-2i\phi_k}} \right)^2 = 1.
\]

We shall prove the formula in three steps:
\begin{enumerate}
\item We show that for all $n$ and $i$ the weight $n$ piece of $H^i(\cH,\cL)$ is finite dimensional and we compute the generating function for the Euler characteristic of $R\Gamma(\cH,\cL)_{\text{weight }n}$, referred to as the graded Euler characteristic $\chi_{\bC^\ast}(\cH,\cL)$ -- See \autoref{prop:index_formula}.
\item We use a version of the ``quantization commutes with reduction" theorem to show that $R\Gamma(\cH,\cL) \simeq R\Gamma(\cH^{ss},\cL)$, and in particular restricting to the open substack of semistable Higgs bundles does not change the graded Euler characteristic - See \autoref{thm:main_index}.
\item We combine a vanishing theorem for higher cohomology $H^i(\cH,\cL)$ for $i>0$ on the stack itself (\autoref{appendix:vanishing}) with the quantization theorem to conclude that the graded index of $R\Gamma(\cH^{\ss},\cL)$ is actually the graded dimension of $H^0(\cH^{\ss},\cL)$ -- See \autoref{thm:main_verlinde}.
\end{enumerate}

Step (1) is not particularly new: The universal deformation of the Verlinde algebra was conjectured in \cite{teleman2004ktheory}, which led to certain index formulas for $K$-theory classes on the stack $\cM_G$, proved in \cite{teleman2009index}. In fact, the generating functions computing the index of $\chi_{\bC^\ast}(\cH,\cL)$ amount to a special case of the generating functions computed in that paper, where we re-interpret the formal parameter $t$ as a geometric parameter coming from the weights of a $\bC^\ast$-action. Our computation in (1) is a straightforward modification of the computations of \cite{teleman2009index}.

The aspect of our methods which are new crucially involve the methods of derived algebraic geometry, and serve as a relatively straightforward sample application for the more general theory of derived $\Theta$-stratifications developed in \cite{halpern2015remarks} and \cite{halpern2016derived} (See also the survey \cite{halpern2016theta}). Specifically, the main structure theorem of \cite{halpern2015remarks} allows one to compare the cohomology of perfect complexes on $\Higgs$ and $\Higgs^\ss$ in steps (2) and (3). The quantization theorem of \cite[Lemma 2.9]{halpern2015remarks} is a generalization of the quantization theorem of \cite{teleman2000quantization} for smooth stacks with $\Theta$-stratification (referred to as a KN-stratification in that case), with the introduction of derived algebraic geometry to deal with the fact that the Harder-Narasimhan strata need not be regularly embedded when the stack is singular.

The remaining input, described in the appendix \autoref{appendix:vanishing}, is the vanishing of $H^{>0}(\cM_G,\cL \otimes \Sym^r(T_{\cM_G}))$ where $T_{\cM_G}$ is the tangent complex of $\cM_G$ and $\cL$ a line bundle having positive level. The proof, which has been known to Teleman for some time, generalizes and simplifies the vanishing theorem proved in \cite{frenkel2013geometric}, and builds on the work of \cites{teleman1998borel,fishel2008strong}.

This work was also inspired by recent results and conjectures in the physics literature: In \cite{gukov2015equivariant}, Gukov and Pei study a 3-dimensional TQFT defined on Seifert manifolds called $\beta$-deformed $G$ complex Chern-Simons theory, and they identify the partition function of this theory on $\Sigma \times S^1$ for the compact group $G_c = U(n)$ with the partition function of a certain ``equivariant $G_c/G_c$ gauged WZW model" \cite[Section 6.1]{gukov2015equivariant}. They conjecture that their partition function with $R$-charge assignment $R=2$ computes the graded dimension of $H^0(\Higgs^{\ss},\cL)$. We formulate the problem in mathematical terms, and prove their conjecture (Although our final answer differs from theirs by an overall factor of $1-t$, See \autoref{ex:gukov_pei}).

\subsection{What's in this paper?}

We prove our main result, the equivariant Verlinde formula in \autoref{thm:main_verlinde}, for the stack of semistable $L$-valued $G$-Higgs bundles on a smooth curve $\Sigma$, where $G$ is any reductive group such that $\pi_1(G)$ is free and $L$ is either $K$ (in genus $g>1$) or any invertible sheaf such that $\deg(L) > \max(0,2g-2)$. We also obtain an equivariant Verlinde \emph{index} formula in \autoref{thm:main_index} which holds for more general $L$. Our set up includes as a special case $L = K(D)$ for an effective divisor $D$, which is sometimes referred to as ``meromorphic Higgs bundles with poles along $D$" \cite{markman1994spectral}, as well as the case $L = (\sqrt{K})^R$, which corresponds to the partition function of the the equivariant $G_c/G_c$-gauged WZW model with arbitrary $R$-charge assignment of the adjoint chiral multiplet as studied in \cite{gukov2015equivariant}.

For any $L$ the stack $\Higgs$ has components corresponding to the topological type of the underlying $G$ bundle, and we have been careful to isolate the contributions coming from each of these connected components. The value of $\dim_{\bC^\ast} H^0(\Higgs^\ss,\cL)$ will only have a contribution from the component corresponding to topologically trivial $G$-bundles, so our equivariant Verlinde formula expresses the value of $\dim_{\bC^\ast} H^0(\Higgs^\ss,\cL(\mu))$, where $\mu$ is a character of $G$ and $\cL(\mu)$ (See \autoref{eqn:special_bundle}) is a twist of $\cL$.  $\cL(\mu)$ only has sections on component\footnote{It is not clear a priori that each connected component of $\Higgs$ is irreducible, and therefore that the projection $\Higgs^\ss \to \cM_G$ induces a bijection on connected components. In the case where $L=K$ and $g>1$ it follows from the results of \cite{beilinson1991quantization}, and for $\deg(L) > \max(0,2g-2)$ it was recently shown in \cite[Proposition 3.2]{arinkin2016irreducible}.} of $\Higgs^\ss$ corresponding to Higgs bundles of a single topological type, determined by $\mu$, and in fact $\cL(\mu)$ descends fractionally to an ample bundle on the good moduli space of that component. Finally, we give some more concrete geometric interpretations, in terms of good moduli spaces and framed bundles, in \S \ref{sect:good_moduli} and \S \ref{sect:framed} respectively.

\subsection{Author's note}

The mark of Constantin Teleman's perspective is indelible throughout this paper, and I would like to thank him for suggesting this problem as a thesis project, and for many helpful conversations over the years. I have also benefited from many helpful conversations with Christopher Woodward, Tony Pantev, and Dima Arinkin on the details of this work. In addition, I have appreciated the support and interest of several of the faculty members at Columbia, including Johan De Jong, Michael Thaddeus, Daniel Litt, Davesh Maulik, and Andrei Okounkov. This research was supported by Columbia University.

While completing this paper I attended a conference\footnote{The lecture was ``Geometric Quantisation of Higgs Moduli Spaces", at the conference ``New Perspectives on Higgs Bundles, Branes, and Quantization," June 13-17, at the Simons Center.} at which J{\o}rgen Ellegaard Andersen announced a proof, joint with Sergei Gukov and Du Pei, of our main theorem. After speaking a bit about the contents of our papers, we decided to complete our projects independently. Ultimately there are some similarities and differences: they go into a lot more depth than our \autoref{rem:parabolic} in the case of parabolic bundles, they treat the case of low genus in detail, and they use their results to construct a 2D TQFT which encodes these index formulas. On the other hand, they only treat simply connected semisimple $G$, and they focus on the line bundles $L = K^{R/2}$. More significantly, they rely on a codimension argument to obtain their dimension formulas on the semistable locus, and thus their results do not include the vanishing of higher cohomology on the semistable locus. I would like to thank them for being so cordial and collegial in coordinating the simultaneous release of our final papers.

\section{Background}

As our results build on the index formulas of \cite{teleman2009index}, we will mostly adopt the notation from that paper. $G$ will denote a (connected) complex reductive group, and we fix a complex maximal torus and Borel subgroup $T \subset B \subset G$. For a finitely generated abelian group $A$, we let $A_\bQ := A \otimes \bQ$, $A_\bC := A \otimes \bC$, and denote the subgroup $A\tw{n} := (2\pi i)^n A \subset A_\bC$. Analytically, we identify $T = N_\bC / N\tw{1}$, where $N$ is the cocharacter lattice of $T$, which we canonically identify with $\pi_1(T)$ as well. In particular for $\xi \in N$, $z \mapsto z^\xi = e^{log(z) \xi}$ is a homomorphism $\bC^\ast \to T$. We let $M$ denote the $\bZ$-dual of $N$, so that $M = \Hom(N\tw{1},\bZ\tw{1}) = \Hom(T,\bC/\bZ(1))$ is the character lattice of $T$, and $T^\dual = M_\bC / M(1)$ is the dual torus. In particular $e^\alpha : T \to \bC^\ast$ is a homomorphism.

\subsubsection{The stack of Higgs bundles}

Fix an invertible sheaf $L$ on $\Sigma$. We denote $\pi : \Sigma \times \cM_G \to \cM_G$ be the projection, $E : \Sigma \times \cM_G \to BG$ the universal bundle, $\fz \subset \fg$ the center, and $\fg' := [\fg,\fg]$. Then we define the perfect complex on $\cM_G$
\begin{equation}\label{eqn:basic_complex}
\cP_L = R\pi_\ast(L^\dual \otimes K \otimes E^\ast(\fg')[1]) \oplus (\cO_\cM \otimes H^0(L)^\dual \otimes \fz).
\end{equation}
Note that $\cP_L$ is connective, so we may regard $\Sym(\cP_L)$ as a connective sheaf of commutative differential graded algebras (CDGA's) on $\cM_G \times B \bG_m$, where $\Sym^k$ has weight $k$ with respect to the $\bG_m$-action. We define the reduced stack of \emph{$L$-valued Higgs bundles}
\[
\Higgs_{G,L} = \iSpec_{\cM_G} (\Sym(\cP_L)).
\]
This is an algebraic derived stack, i.e. a sheaf over the $\infty$-category of connective CDGA's with its \'etale topology.

$\Higgs_{G,L}$ is quasi-smooth as a derived stack, meaning its cotangent complex is concentrated in degree $[-1,0,1]$ (see \S \ref{sect:Euler_complex} below). The underlying classical stack $\Higgs_{G,L}^{\rm{cl}} \hookrightarrow \Higgs_{G,L}$ is $\iSpec$ of $H^0$ of this sheaf of CDGA's, which is typically singular. The two coincide, on any given connected component, when the inequality which holds a priori,
\begin{equation}\label{eqn:classical_condition}
\dim \Higgs_{G,L}^{\rm{cl}} \geq \op{vdim} \Higgs_{G,L} = \dim(\fg) \deg(L) + \dim(\fz) h^1(L),
\end{equation}
is actually an equality. If $S$ is a classical scheme, then using relative Serre duality one can show that a map $S \to \Higgs_{G,L}$, which factors canonically through $\Higgs_{G,L}^{\rm{cl}}$, classifies a $G$-bundle $E : S \times \Sigma \to BG$ along with a section of $(\pi_S)_\ast(L \otimes E^\ast(\fg^\dual))$ (a Higgs field).

\begin{rem} \label{rem:reducedness}
This generalizes the more commonly studied notion of Higgs bundles, which correspond to $L = K$. When $G$ is semisimple $\Higgs_{G,K} = T^\ast \cM_G$, and for an effective divisor $D$, $\Higgs_{G,K(D)}$ is the stack of meromorphic Higgs bundles with poles along $D$ \cites{markman1994spectral}. The naive generalization for reductive $G$ would be $\iSpec(\Sym(R\pi_\ast(L^\dual \otimes K \otimes E^\ast(\fg)[1])))$. This has the issue that the presence of a nontrivial center $\fz \subset \fg$ introduces a trivial summand of $\cP_L$ in homological degree $1$ whenever $h^1(L)>0$. This will cause $\Higgs_{G,L}$ to have a non-trivial derived structure and complicate the ultimate interpretation of our index formula (see \autoref{sect:verlinde_final} for further discussion). Our definition of the ``reduced" stack of $L$-valued $G$-Higgs bundles $\Higgs_{G,L}$ above removes this extra derived structure without changing the underlying classical stack.
\end{rem}

For most of the paper we will work with a derived stack $\cX_{G,L}$ which is the quotient of $\Higgs_{G,L}$ by the action of $\bG_m$ which scales the Higgs field. To be precise, we regard $\Sym(\cP_L)$ as a sheaf of CDGA's on $\cM_G \times B\bG_m$ by letting $\bG_m$ act on $\Sym^k(\cP_L)$ with weight $k$. Then we define
\[
\cX_{G,L} := \iSpec_{\cM_G \times B\bG_m} (\Sym(\cP_L)).
\]
Maps $S \to \cX_{G,L}$ classify invertible sheaves $Q$ on $S$ along with a section of $(\pi_S)_\ast(L \otimes E^\ast(\fg^\dual)) \otimes Q$. When the dependence on $G$ and $L$ is understood, we will drop them from the notation $\cX_{G,L}$, $\Higgs_{G,L}$, and $\cP$.

\subsubsection{Tautological classes}

Consider the projection map $p : \cX_{G,L} \to \cM_G \times B\bC^\ast$. The corresponding pullback map $p^\ast : \Pic(\bung \times B\bC^\ast) \to \Pic(\cX_{G,L})$ is an equivalence, as is the pullback map in cohomology, and topological $K$-theory.\footnote{By this we mean the cohomology and $K$-theory of the topological stack underlying the analytification of the underlying classical stack.} For any of these claims it suffices express $\cM_G$ as an ascending union of open substacks which are global quotients, and over which $\cX_{G,L}$ is the quotient of an affine cone by a scaling action.\footnote{If one defines the algebraic $K$-theory of a locally finite type stack to be the homotopy limit of $K(\Perf(\cU))$ for larger and larger finite type open substacks $\cU$, then $p^\ast$ induces an equivalence in algebraic $K$-theory as well.}

We will consider several tautological perfect complexes on $\cX_{G,L}$: Let $\pi : \Sigma \times \cX_{G,L} \to \cX_{G,L}$ denote the projection, and let $E : \Sigma \times \cX_{G,L} \to BG$ be the tautological $G$ bundle. For a representation $V$ of $G$, we consider $E_x^\ast V := R\pi_\ast(\cO_x \otimes E^\ast(V))$ for a fixed point $x \in \Sigma$, $E_\Sigma^\ast V := R\pi_\ast(\sqrt{K} \otimes E^\ast(V))$, and $D_\Sigma(V) := \det^{-1}(E_\Sigma^\ast(V))$. These are $p^\ast$ of the complexes on $\cM_G$ referred to as ``even Atiyah-Bott classes" in \cite{teleman2009index}, regarded as complexes on $\cM_G \times B\bC^\ast$ with $\bC^\ast$ acting trivially.

We regard an element $h \in H^4(BG;\bQ)$ as an invariant quadratic form on $\fg$, and we say that it is positive if it is positive definite on $N(1)$ (or equivalently, negative definite on $N$). We say that $\cL \in \Pic(\cX_{G,L})$ has \emph{level} $h$ if $c_1(\cL)$ lies in the image of the transgression homomorphism $\tau : H^4(BG;\bQ) \to H^2(\cX_{G,L};\bQ)$, or equivalently if the corresponding invertible sheaf on $\cM_G \times B\bC^\ast$ has level $h$ (after forgetting the action of $B\bC^\ast$) in the sense of \cite{teleman2009index}. The invertible sheaves with positive level are those for which $c_1(L)$ lies in the $\bQ_+$ span of $c_1(D_\Sigma V)$, where $V$ ranges over $\fg$-faithful representations of $G$. If $\cL \in \Pic(\cX_{G,L})$ has a level, the invertible sheaf
\begin{equation}\label{eqn:special_bundle}
\cL(\mu) := \cL \otimes E_x^\ast(\bC_{-\mu}) \in \Pic(\cX)
\end{equation}
features prominently in the sequel.

\subsubsection{Topologial classification of $G$-Higgs bundles}

Because $G$ is connected, the topological type of a principal $G$ bundle on $\Sigma$ is classified by its Chern class in $H^2(\Sigma, \pi_1(G)) \simeq \pi_1(G)$. Let $C \subset Z(G) \subset G$ denote the identity component of the center. Then $\pi_1(C) \to \pi_1(G)$ is a free subgroup of finite index, inducing an equivalence $\pi_1(C)_\bQ \simeq \pi_1(G)_\bQ$.\footnote{Another way to see this equivalence is to consider the short exact sequence of Lie groups $1 \to [G,G] \to G \to G^{\rm{ab}} \to 1$. The resulting short exact sequence $1 \to \pi_1([G,G]) \to \pi_1(G) \to \pi_1(G^{\rm{ab}}) \to 1$ coincides with the quotient of $\pi_1(G)$ by is torsion subgroup. In other words $\pi_1(G)_\bQ \to \pi_1(G^{\rm{ab}})_\bQ$ is an equivalence, so the claim follows from the fact that $C \to G^{\rm{ab}}$ is an isogeny of tori.} Therefore we identify the torsion free quotient $\pi_1(G)_{\rm{tf}}$ with a full-rank lattice of $\pi_1(C)_\bQ$. The inclusion $\pi_1(C) \hookrightarrow \pi_1(T) \simeq N$ gives a canonical embedding $\pi_1(G)_\bQ \hookrightarrow N_\bQ$ which splits the canonical projection $N_\bQ \simeq \pi_1(T)_\bQ \to \pi_1(G)_\bQ$. Alternatively, this splitting can be identified with the canonical decomposition as a $W$-representation $N_\bQ = N_\bQ^W \oplus N_\bQ'$ into trivial and nontrivial pieces. Under this identification $\pi_1(G)_\bQ = N_\bQ^W \subset N$ and $N_\bQ' = \ker(N_\bQ = \pi_1(T)_\bQ \to \pi_1(G)_\bQ)$.

Thus we associate to any principal $G$-bundle an element of $N_\bQ$, which completely determines the topological type of the $G$-bundle if $\pi_1(G)$ is torsion free. In this case we label the connected components of $\cM_{G,\ttype}$ for $\gamma \in \pi_1(G) \subset N_\bQ$. The projection $p : \cX_{G,L} \to \cM_G$ induces a bijection on connected components, so we label them $\cX_{G,L,\ttype}$ as well (again omitting $G$ and $L$ from the notation when they are implied by context).

\subsubsection{Gerbes and rigidification}

The subgroup $C$ is central in $G$, so the canonical action by right multiplication on any $S$-family of $G$-bundles $E$ is actually an action by automorphisms of $G$-bundles. This defines a canonical sub group sheaf $S \times C \hookrightarrow \underline{\op{Aut}}_{\cM}(E)$ over $S$, and it amounts to an embedding of the constant relative group scheme
\[
\cM \times C \hookrightarrow I_{\cM}.
\]
By a general construction (See, for instance \cite[Appendix A]{abramovich2008tame}) one has a ``rigidification," an algebraic stack $\cM \thickslash C$ with a map $\cM \to \cM \thickslash C$ which is a relative gerbe for $C$. Furthermore for any $S$-family of $G$-bundles $E$ as above we have the canonical exact sequence of group schemes over $S$,
\[
\{1\} \to S \times C \to \underline{\op{Aut}}_{\cM}(E) \to \underline{\op{Aut}}_{\cM \thickslash C}(E) \to \{1\},
\]
so points of $\cM \thickslash C$ no longer have canonical positive dimensional subgroups of their automorphism groups. The procedure of rigidification commutes with the formation of good moduli spaces, so the map $\cM_G^{\ss} \to M_G$ factors through $\cM_G^{\ss} \to \cM_G^{\ss} \thickslash C$ and is a good moduli space for the latter.

Any perfect complex on $\bung$ canonically decomposes into a direct sum of perfect complexes of constant weight $\mu \in \widehat{C}$ with respect to the central action of $C$. For any $G$-representation $V$, let $V = \bigoplus_{\mu \in \widehat{C}} V_\mu$ denote its splitting into isotypical pieces under the action of $C$. If $V = V_\mu$ for some $\mu$, then $E_x^\ast V$, $E_\Sigma^\ast V$, and $R\pi_\ast(L\otimes E^\ast V)$ for any $L \in \Pic(\Sigma)$ are concentrated in $C$-weight $\mu$ at every $S$-point of $\bung$. This implies that for any $S$-family of $G$-bundles $E$, the induced $C$-action on the sheaves $(\pi_S)_\ast(L \otimes E^\ast \fg^\dual)$ and $(\pi_S)_\ast(L \otimes E^\ast \fg^\dual) \otimes Q$ is trivial, so we have embeddings $\Higgs \times C \hookrightarrow I_{\Higgs}$ and $\cX \times C \hookrightarrow I_{\cX}$ just as with $\cM$.

\begin{rem} \label{rem:cotangent_stack}
This gives another perspective on the derived moduli of reduced Higgs bundles in the case $L = K$. The complex $\cP_K$ has $C$ weight $0$ and thus descends uniquely to $\cM_G \thickslash C$.\footnote{This can be checked \'etale locally over $\cM_G\thickslash C$, so we may reduce this to the case of a trivial gerbe, which is easy to check directly.} Regarding $\cP_G$ as a complex on $\cM_G \thickslash C$ one has $\cP_K \simeq (L_{\cM\thickslash C})^\dual$ -- removing the action of $C$ has the effect of removing the trivial summand $\cO \otimes \fz[1]$ from $L_{\cM_G}^\dual$. It follows that $\Higgs^{\op{cl}}_{G,K} \thickslash C \simeq T^\ast (\cM_G \thickslash C)^{\op{cl}}$, and likewise $\cX^{\op{cl}}_{G,K} \thickslash C \simeq T^\ast(\cM_G \thickslash C)^{\op{cl}} / \bG_m$ with the right hand side interpreted as above. We are not aware of a reference for the rigidification procedure in the context of derived stacks, but observe that $\Higgs_{G,K} \simeq \cM_G \times_{\cM_G \thickslash C} T^\ast(\cM_G \thickslash C)$ as a derived stack, so the map $\Higgs_{G,K} \to T^\ast(\cM_G \thickslash C)$ is still a $C$-gerbe in the derived setting, and likewise for $\cX_{G,K}$. We can thus regard $\Higgs_{G,K} = T^\ast(\cM_G \thickslash C)$ heuristically (pending a formal treatment of rigidification in derived algebraic geometry).
\end{rem}

\subsubsection{Harder-Narasimhan stratification}

The stack $\Higgs$ has a stratification whose open stratum $\Higgs^{\ss} \subset \Higgs$ classifies semistable $L$-valued Higgs bundles. $\Higgs^{\rm{cl},\ss}$ admits a good quotient $\Higgs^{\rm{cl},\ss} \to \op{H}$ whose connected components are projective-over-affine. See \cite{simpson1994modulib} for a construction with $L = K$ and the components of $\cX$ with trivial rational Chern classes, and \cite{nitsure1991moduli, simpson1994modulia} for a construction with general $L$ and $G=\GL_n$.

The complement of $\Higgs_{G,L}^{\ss}$, the locus of unstable bundles, is the union of locally closed strata classifying Higgs bundles of fixed Harder-Narasimhan type. More precisely, any unstable Higgs bundle $(E,\phi)$ admits a canonical reduction of structure group $(E',\phi')$ to a parabolic $P \subset G$, called the Harder-Narasimhan reduction \cite{dey2005harder}. The topological type of $E'$ is determined by an element of $\pi_1(P) = \pi_1(P/U)$, where $U$ is the unipotent radical, which we can identify as above with an element $\xi \in N_\bQ$, which will be dominant. Then $P = P_\xi$ will be the standard parabolic subgroup defined by $\xi$, and we denote $G_\xi \subset G$ the corresponding Levi subgroup. Thus we denote the (set theoretic) stratification
\begin{equation} \label{eqn:HN_stratification}
\Higgs_{G,L} = \Higgs_{G,L}^{\ss} \cup \bigcup_{\xi} \cS_{\xi}^{\Higgs},
\end{equation}
where the locally closed strata $\cS^{\Higgs}_\xi$ are indexed by those dominant $\xi \in N_\bQ$ which lie in the subgroup $\pi_1(P_\xi) \subset N_\bQ$. The stratification of $\Higgs_{G,L}$ is $\bC^\ast$-invariant, and so descends to a stratification $\cX_{G,L} = \cX_{G,L}^{\ss} \cup \bigcup_\xi \cS_\xi$.

\subsubsection{Harder-Narasimhan stratification as a $\Theta$-stratification}

The Harder-Narasimhan stratification of $\Higgs_{G,L}$ and $\cX_{G,L}$ are $\Theta$-stratifications in the sense of \cite{halpern2014structure}, meaning that the strata admit modular interpretations as open substacks of the mapping stack $\iMap(\Theta,\Higgs_{G,L})$ and $\iMap(\Theta,\cX_{G,L})$ respectively, with $\Theta := \bA^1/\bG_m$. This observation, along with the algebraicity of this mapping stack in the derived context established in \cite{halpern2014mapping}, is crucial to the results of this paper, in that it equips $S_\xi$ canonically with the structure of a locally closed derived subscheme, which need not be classical even when $\Higgs$ is.

Let us describe the mapping stack $\iMap(\Theta,\cX_{G,L})$ explicitly: Any invertible sheaf on $\Theta$ must be of the form $\cO_\Theta \langle k \rangle$ for some $k \in \bZ$, and a $G$-bundle $\Theta \times \Sigma$ corresponds to a $\xi \in N$ and principal $P_\xi$-bundle $E$ for the standard parabolic $P_\xi \subset G$ (see \cite{halpern2014structure}). We can consider the weighted descending filtration of $\fg$ where $\fg_{\xi \geq p} \subset \fg$ is the subspace spanned by eigenvectors of $\xi$ of eigenvalue $\geq p$. This is a filtration as a $P_\xi$-representation, so we have a filtrations of $E^\ast(\fg)$, which is the adjoint bundle of the associated $G$-bundle $E \times_{P_\xi} G$, by the subbundles $E^\ast(\fg_{\xi \geq p})$.

A map $\Theta \to \cX_{G,L}$ therefore consists of a choice of $k \in \bZ$, $\xi \in N$, a $P_\xi$-bundle $E$, and a section of $\Gamma(\Sigma, L \otimes E^\ast(\fg_{\xi \geq k}))$. Likewise the groupoid of maps $S \times \Theta \to \cX_{G,L}$ for a connected scheme $S$ is equivalent to the groupoid consisting of $(k,\xi) \in \bZ \times N$, an invertible sheaf $Q$ on $S$, a $P_\xi$-bundle $E$ on $S \times \Sigma$, and a section of $\pi_\ast(L \otimes E^\ast(\fg_{\xi \geq k})) \otimes Q$ on $S$. The discrete invariants $(k,\xi)$, along with the topological type of the $P_\xi$-bundle $E$, index the connected components of the mapping stack $\iMap(\Theta,\cX_{G,L})$. The mapping stack $\iMap(\Theta,\Higgs_{G,L})$ has the same description, except that $Q$ is trivialized and $k=0$.

Say $(E,\phi) \in \cX_{G,L}$ is an unstable $G$-Higgs bundle with Harder-Narasimhan type $\xi \in N_\bQ$. Choosing a positive $n \in \bZ$ such that $n \xi \in N$, the cocharacter $n \xi$ along with the Harder-Narasimhan reduction to $P_{n\xi} = P_\xi \subset G$ correspond to a $G$-bundle on $\Theta \times \Sigma$, and the fact that $\phi$ is compatible with this reduction means that $\phi$ lies in $H^0(\Sigma, L \otimes E^\ast(\fg_{\xi \geq 0})) \subset H^0(\Sigma, L \otimes E^\ast(\fg))$. Hence the Harder-Narasimhan reduction gives the data of a map $f : \Theta \to \cX_{G,L}$ along with an isomorphism $f(1) \simeq (E,\phi)$, canonical up to the choice of multiplier $n$ corresponding ramified covers of $\Theta$. This description works in families as well.

The \emph{centers} of the $\Theta$-strata, denoted $\cZ_\xi^{\ss}$, are open substacks of $\iMap(B\bC^\ast,\cX_{G,L})$. They also have explicit modular interpretations in this case: they are equivalent to $\cX_{G_\xi, L}^{\ss}$, where $G_\xi$ denotes the centralizer of $\xi$ (i.e. the Levi subgroup associated to $\xi$). The map $\sigma_\xi : \cZ_\xi^{\ss} \to \cX_{G,L}$, factoring through $\cS_\xi$, is the map which induces a family of $G$-Higgs bundles from a family of $G_\xi$-Higgs bundles via the inclusion $G_\xi \hookrightarrow G$. The projection $\cS_\xi \to \cZ_\xi^{\ss}$ is the induction map for the homomorphism $P_\xi \to G_\xi$.

$\cZ_\xi$ is a $\bG_m$-gerbe, as $\xi$ itself defines a homomorphism $\xi_S : \bG_m \times S \to \op{Aut}_S(E,\phi)$ for every $S$-family of $G_\xi$-Higgs $(E,\phi)$. Thus every complex in $\QCoh(\cZ_\xi)$ decomposes functorially into a direct sum of complexes on which $\xi$ acts with constant weight, and we use the term \emph{weights of $F \in \QCoh(\cZ_\xi)$} to refer to the weights in which these summand on non-vanishing. We sometimes use $F^{>0}$ and $F^{\leq 0}$ to denote the direct summand concentrated in weight $>0$ and $\leq 0$ respectively. After restricting to $\cZ_\xi^\ss$, \cite[Lemma 2.4]{halpern2015remarks} identifies the exact triangles of cotangent complexes,
\[
\xymatrix@R=10pt{
L_{\cS_\xi/ \cX}[-1]|_{\cZ_\xi^\ss} \ar[d]^{\simeq} \ar[r] & L_\cX|_{\cZ_\xi^\ss} \ar[d]^{\simeq} \ar[r] & L_{\cS_\xi}|_{\cZ_\xi^\ss} \ar[d]^{\simeq} \ar[r] & \\
L_{\cX}|_{\cZ_{\xi}^\ss}^{>0} \ar[r] & L_{\cX}|_{\cZ_{\xi}^\ss} \ar[r] & L_{\cX}|_{\cZ_{\xi}^\ss}^{\leq 0} \ar[r] & },
\]
the top row refers to the canonical exact triangle for the inclusion $\cS_\xi \hookrightarrow \cX$ with the derived structure discussed above, and the bottom from is the canonical decomposition of $L_\cX|_{\cZ_\xi^\ss}$ into positive and nonnegative weight pieces. This identification fails if instead we used the underlying classical stack $\cS^{\rm{cl}}_\xi$.

\section{Index formulas on $\cX$}

\subsection{Recollections on non-abelian localization}

Given a $\Theta$-stratification of a quasi-smooth derived stack $\cY = \cY^{ss} \cup \cS_\beta$ indexed by some totally ordered set $\beta \in I$, we recall the non-abelian localization theorem of \cite{halpern2015remarks}. We define $L^+_\beta \in \Perf(\cZ_\beta^{\ss})$ to be the direct summand of $L_\cY|_{\cZ_\beta^{\ss}}$ with homology concentrated in positive weights, and likewise $L^-_\beta$ is the summand with negative weights. Consider the complexes
\[
\bE_\beta = \Sym(L_\beta^- \oplus (L_\beta^+)^\dual) \otimes \det(L_\beta^+)^\dual [-\rank(L_\beta^+)] \in \QCoh(\cZ_\beta^\ss)
\]
We define $\eta_\beta$ to be the weight of $\det(L_\beta^+)$. Note that the complex is concentrated in weight $\leq \eta_\beta$, and the summand of $\bE_\beta$ in any fixed weight is a perfect complex. We shall say that a complex of vector spaces is ``finite dimensional" if the direct sum of all homology groups is finite dimensional as a $k$-vector space.

\begin{thm}[Virtual non-abelian localization] \label{thm:virtual_localization} \cite{halpern2015remarks}
Let $\cY$ be a quasi-smooth derived stack with a $\Theta$-stratification $\cY = \cY^{\ss} \cup \bigcup_\alpha \cS_\beta$. Assume that $F \in \Perf(\cY)$ is such that $R\Gamma(\cZ_\beta^\ss,F|_{\cZ_\beta^\ss} \otimes \bE_\beta)$ is finite dimensional for all $\beta$ and acyclic for all but finitely many $\beta$. Then $R\Gamma(\cY,F)$ is finite dimensional if and only if $R\Gamma(\cY^{\ss},F)$ is, and we have
\[
\chi(\cY,F) = \chi(\cY^{\ss},F) + \sum_{\beta} \chi(\cZ_\beta^{\ss},F|_{\cZ_\beta^{\ss}} \otimes \bE_\beta).
\]
\end{thm}

If $\cY^\ss$ and $\cZ_\beta^\ss$ are finite type with proper good moduli spaces, then $R\Gamma(\cY^{\ss},F)$ and $R\Gamma(\cZ_\beta^{\ss},F|_{\cZ_\beta^\ss} \otimes \bE_\beta)$ are finite dimensional for any $F \in \Perf(\cY)$, substantially simplifying the statement.\footnote{More generally this holds as long as $\cY^{\ss}$ and $\cZ_\beta^{\ss}$ are cohomologically proper in the sense of \cite{halpern2014mapping}. This condition holds for the stack $\cX_{G,L}$ (See \autoref{sect:good_moduli}).} We shall see that typically $R\Gamma(\cZ_\beta^{\ss},F|_{\cZ_\beta^\ss} \otimes \bE_\beta)$ is acyclic because the weights of $F|_{\cZ_\beta^\ss}$ are $<\eta_\beta$. Along those lines we will also use a slightly stronger result from \cite{halpern2015remarks}, with $\cY$ as above
\begin{prop}[quantization commutes with reduction]
If $G \in \Perf(\cY)$ is such that $H_\ast(G|_{\cZ_\beta^\ss})$ is supported in weights $<\eta_\beta$ for all $\beta$, then the restriction map
\[
R\Gamma(\cY,G) \to R\Gamma(\cY^\ss,G)
\]
is an equivalence.
\end{prop}
\begin{proof}
It suffices to prove this for a single closed stratum. This is \cite[Lemma 2.9]{halpern2015remarks} with $F= \cO_\cY$, which lies in $D^-\Coh(\cY)^{\geq 0}$ by definition. In order to apply that lemma, we must observe that a \emph{perfect} complex $G$ lies in $D^-\Coh(\cY)^{<0}$ if and only if the weights of $G|_{\cZ_\beta^\ss}$ are $<\eta_\beta$ for all $\beta$, which follows from \cite[Remark 3.11]{halpern2015remarks}.
\end{proof}

Define a complex $F \in \Perf(\cY)$ to be
\begin{itemize}
\item \emph{almost admissible} if the weights of $F|_{\cZ_\beta^\ss}$ are $<\eta_\beta$ for all but finitely many $\beta$, and
\item \emph{$\cL$-admissible} with respect to $\cL \in \Pic(\cY)$ if $\cL \otimes F^{\otimes m}$ is almost admissible for any $m>0$.
\end{itemize}
We refer to the category of $\cL$-admissible complexes as $\Perf(\cY)^{\cL\rm{-adm}}$. This category has some convenient properties
\begin{lem} \label{lem:admissible}
$\Perf(\cY)^{\cL\rm{-adm}}$ is closed under shifts, cones, summands, and tensor products, and symmetric powers.
\end{lem}
\begin{proof}
For $F\in \Perf(\cZ_\alpha^\ss)$, we let $\hWt$ denote the highest weight of a non-zero homology sheaf $H_i(F)$. One can show that $\hWt(F^{\otimes m}) = m \hWt(F)$, so $F\in \Perf(\cY)^{\cL\rm{-adm}}$ if and only if for all $m\geq 1$
\[
\hWt(F|_{\cZ_\beta^\ss}) < \frac{1}{m}\left(\eta_\beta - \op{wt}(\cL|_{\cZ_\alpha^\ss})\right) \text{ for almost all } \beta.
\]
From this characterization of $\Perf(\cY)^{\cL\rm{-adm}}$, the claim follows easily from the following properties, which are loosely analogous to the properties of a valuation on a ring:
\begin{gather*}
\hWt(\op{Cone}(E \to F)) \leq \max(\hWt(E),\hWt(F)); \\
\hWt(E \otimes F) \leq \hWt(E) + \hWt(F).
\end{gather*}
\end{proof}

As discussed above, if $\cY^\ss$ and $\cZ_\beta^\ss$ are cohomologically proper in the sense of \cite{halpern2014mapping}, which is the case for the $\Theta$-stratifications of $\cM_{\ttype}$ and $\cX_{\ttype}$ by \autoref{rem:properness}, then $R\Gamma(\cY,F)$ is finite dimensional for any almost admissible complex $F$. We also note that because $\Perf(\cY)^{\cL\rm{-adm}}$ is a stable symmetric monoidal $\infty$-subcategory, $K^0(\Perf(\cY)^{\cL\rm{-adm}})$ is a $\lambda$-ring and the map $K^0(\Perf(\cY)^{\cL\rm{-adm}}) \to K^0(\Perf(\cY))$ is a map of $\lambda$-rings.

Specializing to the case where $\cY = \cM_G$: If $\cL \in \Pic(\cM_G)$ has level $h$ then the $\xi$-weight of $\cL|_{\cZ_\xi^\ss}$ is $h(\xi,\xi)$ (see \autoref{lem:weights}), which is negative and grows with order $|\xi|^2$. Furthermore $\eta_\xi = -c(\xi,\xi) + O(|\xi|)$ on $\cM_G$, so $F \in \Perf(\cM_G)^{\cL\rm{-adm}}$ if for any $m>0$
\[
\hWt_\xi (F|_{\cM_{G_\xi}^\ss}) \leq \frac{1}{m} (h+c)(\xi,\xi) + o(|\xi|^2).
\]
In fact, this implies that $F$ is $\cL$-admissible for some $\cL$ of level $h>-c$ if and only if it is $\cL$-admissible for any such $\cL$. By the weight computations of \cite{teleman2009index}, the Atiyah-Bott complexes $E_x^\ast V$ and $E_\Sigma^\ast V$ admit upper bounds of order $O(|\xi|)$ on the $\xi$-weights of their restrictions to $\cM_{G_\xi}^\ss$, hence they are $\cL$-admissible for any $\cL$ of level $h>-c$.

\begin{lem}
Let $\cL \in \Pic(\cM_G)$ have level $h>-c$, and let $F \in \Perf(\cM_G \times B\bG_m)$ be a complex such that 1) $F$ is concentrated in finitely many weights with respect to $\bG_m$, and 2) $F \in \Perf(\cM_G)^{\cL\rm{-adm}}$ after forgetting the action of $\bG_m$. Then $R\Gamma(\cX_{G,L}, p^\ast(F \otimes \cL))$ is finite dimensional.
\end{lem}

\begin{proof}
By the projection formula $R\Gamma(\cX, p^\ast(F \otimes \cL)) \simeq R\Gamma(\bung \times B\bG_m,F \otimes \cL \otimes \Sym(\cP))$. This will split into a finite direct sum of complexes of the form $R\Gamma(\bung,F_k \otimes \cL \otimes \Sym^k(\cP))$, where $F_k$ is summand of $F$ of $\bG_m$-weight $k$, and finite dimensionality follows from the $\cL$-admissibility of $F_k \otimes \Sym^k(\cP)$, which we check with \autoref{lem:admissible}. $F_k$ is $\cL$-admissible because it is a summand of $F$. The $\xi$-weights of the complex $\cP$ admit a bound of order $O(|\xi|)$ by the weight computations of \cite{teleman2009index}, thus $\cP^{\otimes k}$ and $\op{Sym}^k(\cP)$ are $\cL$-admissible.
\end{proof}

\begin{rem}
In fact any $F \in \Perf(\cX)$ whose restriction to $\bung \times B\bG_m$ is concentrated in finitely many weights with respect to $\bG_m$ admits a finite filtration whose associated graded is $p^\ast(F|_{\bung \times B\bG_m})$. So the conclusion of the previous lemma applies to any $F \in \Perf(\cX)$ such that $F|_{\bung \times B\bG_m}$ satisfies the stated hypotheses.
\end{rem}

\subsubsection{Algebraic equivalence on $K$-theory}

Note that our definition of $\cL$-admissible complex is slightly more general than the notion of admissible classes in topological $K$-theory $K^0_{top}(\cM_G)$, defined in \cite{teleman2009index} as the span of Atiyah-Bott generators $E_x^\ast V$ and $E_\Sigma^\ast V$ tensored with an $\cL$ with level exceeding $c$, the level of $\cK^{1/2}$. It follows a posteriori from their explicit index formulas that for these complexes the algebraic index $\chi(\cM_G,-)$ depends only on the underlying topological $K$-theory class. We expect that for $\cL$-admissible complexes in our more general sense, $\chi(\cY,\cL \otimes F)$ only depends on the topological $K$-theory class of $F$, but we will get away with a slightly more naive formulation:

We shall consider the equivalence relation $\sim$ of algebraic equivalence on algebraic $K$-theory in three contexts: on $K^0(\Perf(X))$ when $X$ is a scheme, for $K^0$ of the category of almost admissible complexes on a quasi-smooth $\cY$ with $\Theta$-stratification, and for $K^0(\Perf(\cY)^{\cL\rm{-adm}})$ in the same situation. By definition the relation on $K^0(\Perf(\cY)^{\cL\rm{-adm}})$ is generated by the relation $F_0 \sim F_1$ if there is a connected smooth scheme $S$ with points $0,1\in S$ and $F \in \Perf(S \times \cY)$ which is $\cL$-admissible with respect to the induced $\Theta$-stratification of $S \times \cY$ and such that $F|_{\{0\}} \simeq F_0$ and $F|_{\{1\}} \simeq F_1$.

Algebraic equivalence in the category of almost admissible complexes on $\cY$ or on $K^0(\Perf(X))$ (regarded as having a trivial $\Theta$-stratification) is defined analogously. It is compatible with tensor products and $\lambda$-operations in $K^0(\Perf(X))$ and $K^0(\Perf(\cY)^{\cL\rm{-adm}})$, and the map $F \mapsto \cL \otimes F$ from almost admissible complexes to $\Perf(\cY)^{\cL\rm{-adm}}$ is compatible with algebraic equivalence. If $\cY^{\ss}$ and $\cZ_\beta^{\ss}$ admit proper good moduli spaces (or are cohomologically proper) then \autoref{thm:virtual_localization} implies that the index $\chi(\cY,-)$ descends to the quotient $K^0(\Perf(\cY)^{\cL\rm{-adm}}) / \sim$. For brevity we denote $K_{\rm{adm}}^0(\cM) = K^0(\Perf(\cM)^{\cL-\rm{adm}}) / \sim$ for any $\cL$ of level $h>-c$.

\subsection{The index formula}
\label{sect:index_formula}

First let us recall the set up of \cite{teleman2009index}. We fix a level $h \in H^4(BG;\bQ)$ such that $h' = h+c$ is a negative definite integer-valued quadratic form on $N$. Then $\xi \mapsto \iota(\xi) h'$ defines a homomorphism $N \to M$ which descends to an isogeny $\chi' : T \to T^\dual$. Note that for a weight $\alpha \in M$, we can define a formal map $1-t e^\alpha : T \to \bC$, and regard $(1-te^\alpha)^\alpha : T \to T^\dual$. We define
\[
\chi'_t = \chi' \cdot \prod_{\text{roots } \alpha>0} \left[ \frac{1-t e^\alpha}{1-t e^{-\alpha}} \right]^\alpha : T \to T^\dual
\]
We let $F_{h,t}$ denote the set of formal points of $T$ which solve the equation $\chi'_t(f) = e^{2\pi i \rho} \in T^\dual$, the generalized Bethe ansatz equation, and $F_{h,t}^{reg}$ denotes solutions which are regular $G$-conjugacy classes at $t=0$. Let $H(f) : \ft \to \ft^\dual$ denote the differential of $\chi'_t$ at $f \in T$, and convert this to a map $H^\dagger(f) : \ft \to \ft$ using $h'$. Finally, define
\begin{equation}\label{eqn:define_theta}
\theta_t(f) := \frac{\prod_{\text{\rm{roots }}\alpha} (1-e^\alpha(f))}{|F| \cdot \det (H^\dagger(f))}.
\end{equation}

\begin{ex}
For $G = SL_2$ both $N \simeq M \simeq \bZ$, and $T = \bC / \bZ(1) \simeq \bC^\ast$. The level $-\op{Tr}_{\bC^2}$ corresponds to the positive generator $\cO(1)$ of $\Pic(\cX_{G,L})$. For the level corresponding to $\cO(h)$, the map $\chi' : \bC^\ast \to \bC^\ast$ is $\chi'(f) = f^{-2(h+2)}$ for $f \in \bC^\ast$. We have
\begin{align*}
\chi_t'(f) &= f^{-2(h+2)} \left( \frac{1-t f^2}{1-tf^{-2}} \right)^2 \\
H^\dagger(f) &=  \chi'_t(f) \cdot \left( 1 + \frac{2t}{h+2} \left( \frac{f^2 - 2t + f^{-2}}{(1-tf^2)(1-tf^{-2})} \right) \right) \\
\theta_t(f)^{-1} &= \frac{2(h+2)}{2-f^2 - f^{-2}} \cdot H^\dagger(f)
\end{align*}
$F_{h,t}$ consists of formal solutions of $\chi_t'(f_t) = 1$, which are uniquely determined by their leading term $f_0$, which is a $2(h+2)$-root of unity. The Weyl group action on the set of solutions is $f_t \mapsto f_t^{-1}$, so there are $h+1$ regular solutions up to the action of $W$. Under the change of coordinates $f = e^{i\phi}$, we have $f^2 + f^{-2} = 2- 4 \sin^2(\phi)$ and $(1-tf^2)(1-tf^{-2}) = (1-t)^2 + 4t \sin^2(\phi)$. Re-expressing these equations in terms of $\phi$ leads to the formula \eqref{eqn:ex_SL2} in the introduction.
\end{ex}

\begin{ex}
In the case of $GL_2$, we identify a point in the maximal torus $T$ by its diagonal entries $(u,v) = (e^{i\tau},e^{i\sigma}) \in (\bC^\ast)^2$. A positive level corresponds to a quadratic form on the Lie algebra of the form $-h_1 (\tau^2 + \sigma^2) - 2 h_2 \tau \sigma$ where $h_1$ and $h_2$ are positive integers and $h_1>h_2$. In this case the Bethe ansatz equation is $\chi_t'(u,v) = (-1,-1)$, which amounts to
\[
u^{h_1} v^{h_2} (1-t\frac{u}{v}) (1-t \frac{v}{u})^{-1} = -1 \qquad u^{h_2} v^{h_1} (1-t\frac{u}{v})^{-1} (1-t \frac{v}{u}) = -1
\]
We leave it as an exercise for the reader to compute the Jacobian determinant of $\chi_t'(u,v)$ with respect to the coordinates $\tau$ and $\sigma$ in order to find $\theta_t(u,v)$.
\end{ex}

\begin{prop} \label{prop:index_formula}
Let $\cL \in \Pic(\bung)$ have level $h > -c$, let $\cL\langle n \rangle \in \Pic(\cM \times B\bG_m)$ denote the invertible sheaf $\cL$ concentrated in weight $-n$ with respect to $\bG_m$, and let $U$ be a representation of $G$. Then we have
\begin{equation} \label{eqn:index}
\sum_{n \in \bZ} t^n \chi(\cX_{G,L}, p^\ast(\cL\langle n \rangle \otimes E_x^\ast U)) = (1-t)^{\#_L} \sum_{f_t \in F_{h,t}^{reg}} \left(\prod_\alpha (1-t e^\alpha(f_t))\right)^{-\chi(L)} \theta_t(f_t)^{1-g} \op{Tr}_U(f_t),
\end{equation}
with $f_t \in F_{h,t}^{reg}$ ranging over a complete set of Weyl orbits, and where $\chi(L) = \deg(L)+1-g$ is the Euler characteristic and $\#_L := -\rank(\fg)\chi(L) -\dim(\fz) h^1(L)$.
\end{prop}
\begin{proof}
By the projection formula the power series above is the index of the series of admissible complexes
\[
\sum_{n \geq 0} t^n \Sym^n(\cP_L) \otimes \cL \otimes E_x^\ast U,
\]
which we can rewrite in $K^0_{\rm{adm}}(\cM)[[t]]$ as $\lambda_{-t}(\cP_L[1]) \otimes \cL \otimes E_x^\ast U$.

We can compute this index by a straightforward modification of \cite[Theorem 6.4]{teleman2009index}: We rewrite $L^\dual \otimes K \sim \sqrt{K} + (g-1 - \deg(L)) \cO_x$ in $K^0(\Perf(\Sigma))$. The compatibility of $R\pi_\ast((-) \otimes E^\ast(V)) : K^0(\Perf(\Sigma)) \to K^0(\Perf(\cM_G))$ with algebraic equivalence implies that
\[
\cP_L \sim (\deg(L)+1-g) E_x \fg' - E_\Sigma \fg' + \dim(\fz) h^0(L) \cO_\cM
\]
as $\cL$-admissible complexes, and $E_\Sigma \fg' = E_\Sigma \fg$ in $K$-theory, so
\[
\lambda_{-t}(\cP_L[1]) \sim \lambda_{-t}(E_\Sigma \fg) \otimes \lambda_{-t}(E_x \fg')^{\otimes(g-1-\deg(L))} \otimes (1-t)^{-\dim(\fz) h^0(L)}
\]
in each $t$-degree.

By expressing the class $\lambda_{-t}$ in terms of Adams operations $\psi^p$ as in \cite{teleman2009index} and using the fact that $\psi^p E_\Sigma^\ast V = \frac{1}{n} E_\Sigma^\ast(\psi^p V)$, we rewrite the index $\chi(\cM,\lambda_{-t}(\cP_L[1]) \otimes \cL \otimes E_x^\ast U)$ as 
\[
(1-t)^{-\dim(\fz) h^0(L)} \chi\left(\cM,\exp\left[- \sum_{p>0} \frac{t^p}{p^2} E_\Sigma(\psi^p(\fg))\right] \otimes E_x^\ast(\lambda_{-t}(\fg')^{\otimes(g-1-\deg(L))}) \otimes \cL \otimes E_x^\ast U \right).
\]
The formula now follows from the main index formula of \cite[Theorem 2.15]{teleman2009index}. The prefactor of $(\prod_\alpha (1-t e^\alpha))^{g-1-\deg(L)}$ in the final answer comes from the factor $E_x^\ast(\lambda_{-t}(\fg')^{\otimes(g-1-\deg(L))})$.
\end{proof}

\begin{rem}
This computation can be readily modified to incorporate other even Atiyah-Bott generators as in \cite[Theorem 6.4]{teleman2009index}, or odd Atiyah-Bott generators.
\end{rem}

\begin{rem}[Parabolic Higgs bundles] \label{rem:parabolic}
In the notation of \cite{teleman2009index}, we may consider the stack $\cM(x,\mathbf{B})$ of $G$-bundles along with a reduction of structure group to $B$ at the point $x$ (i.e. parabolic bundles for the full flag parabolic structure). This projects $p : \cM(x,\mathbf{B}) \to \cM$ as a $G/B$ flag bundle, and there is a natural line bundle on $\cM(x,\mathbf{B})$ extending the weight line bundles $\cO(\mu)$ on the fibers of $p$. We let $\cL\{\mu\}$ denote the twist by this line bundle. There are equivariant index formulas on the cotangent stack $T^\ast \cM(x,\mathbf{B})$, i.e. the stack of parabolic Higgs bundles, analogous to the one computed above:

Consider the graded index of $\cL\{\mu\} \otimes E_x^\ast U$ pulled back to $T^\ast \cM(x,\mathbf{B})$, defined as a generating function for the index on the stack $T^\ast \cM(x,\mathbf{B}) / \bG_m$ as in \autoref{prop:index_formula}. This amounts to computing
\[
\chi_{\bC^\ast}(T^\ast \cM(x,\mathbf{B}), \cL\{\mu\} \otimes E_x^\ast U) := \sum_n t^n \chi(\cM(x,\mathbf{B}),\Sym^n(T_{\cM(x,\mathbf{B})}) \otimes \cL\{\mu\} \otimes E_x^\ast U),
\]
where $T_{\cM(x,\mathbf{B})}$ is the tangent complex. The canonical exact triangle for tangent complexes gives the identity $\lambda_{-t}(p^\ast(T_{\cM}[1])) \otimes \lambda_{-t}(T_{p}[1])$, where $T_p$ is the relative tangent bundle. Therefore this index is equal to the index of $p_\ast(\lambda_{-t}(T_{p}[1]) \otimes \cO\{\mu\}) \otimes \lambda_{-t}(T_{\cM}[1]) \otimes \cL$ on $\cM$ itself via the projection formula, and the result is
\begin{equation}\label{eqn:parabolic}
(1-t)^{-\rank(\fg) \chi(L)} \sum_{f_t \in F_{h,t}^{reg}} \left(\prod_\alpha (1-t e^\alpha(f_t))\right)^{-\chi(L)} \theta_t(f_t)^{1-g} \frac{e^\mu}{\prod \limits_{\alpha >0} (1-te^{-\alpha})(1-e^\alpha)} \op{Tr}_U(f_t),
\end{equation}
where the sum is taken over all of $F_{h,t}^{reg}$ instead of Weyl orbits. The factors of $1-te^{-\alpha}$ and $e^\mu$ account for the weights of $\lambda_{-t}(T_\pi [1]) \otimes \cO\{\mu\}$, whereas the factors of $1-e^\alpha$ and the summation over $W$ come from the Weyl character formula.\footnote{This is a purely formal consequence of the $W$ anti-invariance of the numerator and denominator in the Weyl character formula. If $V_\mu$ is the highest weight representation corresponding to a character $\mu$ of $B$ (with a homological degree shift, if $\mu$ is not dominant), $F \subset T$ is a Weyl-invariant set of points, and $\theta(f)$ a $W$-invariant function on $T$, then $\sum_{F/W} \theta(f) \op{Tr}_{V_\mu}(f) = \sum_F \theta(f) \frac{e^\mu(f)}{\prod \limits_{\alpha>0} (1-e^\alpha(f))}$.}

Our methods can be extended to treat the case of $T^\ast \cM(x,\mathbf{B})$, and more generally $T^\ast \cM(x,\mathbf{P})$ for any quasi-parabolic datum $\mathbf{P}$ (See \cite[\S 9]{teleman2000quantization}). In this setting the $\Theta$-stratification of $T^\ast \cM(x,\mathbf{P})$, and even the semistable locus, depends on a choice of stability condition, encoded in the choice of a positive line bundle $\cL$ on $\cM(x,\mathbf{P})$. For example on $\cM(x,\mathbf{B})$ this is a line bundle of the form $\cL\{\mu\}$ above, so it is specified by a positive line bundle on $\cM$ and a weight of $B$. In principal, the index formula \eqref{eqn:parabolic}, when $U$ is trivial, can be interpreted as the index of such a positive $\cL\{\mu\}$ on $T^\ast \cM(x,\mathbf{P})^\ss$ where semistability is taken with respect to the same $\cL\{\mu\}$. As the analysis of the stratification is a bit more delicate, we postpone a full discussion for future work.
\end{rem}

\section{The equivariant Verlinde formula}

We now use \autoref{thm:virtual_localization} combined with the index formula of \autoref{prop:index_formula} to establish formulas for the equivariant index of certain positive bundles on the semistable locus $\cX_{G,L,\ttype}^\ss$. 

\subsection{The Euler complexes for the moduli of Higgs bundles}
\label{sect:Euler_complex}

We now compute $\bE_\xi$ for the $\Theta$-stratification of $\cX_{G,L}$. Specifically we would like to compute the highest weight $-\eta_\xi$ appearing in $\bE_\xi$. As $\cX \simeq \iSpec_{\cM \times B\bG_m}(\Sym(\cP))$, we have that $L_{\cX/\cM} \simeq p^\ast \cP$, and we can express $L_\cX$ as
\begin{equation} \label{eqn:cotangent_complex}
L_\cX \simeq \Cone(p^\ast(\cP) [-1] \to p^\ast (L_{\cM \times B\bG_m})),
\end{equation}
where $p : \cX \to \bung \times B\bG_m$ is the projection. It follows that
\[
\det(L_\xi^+) := \det(L_\cX|_{\cZ_\xi^\ss}^{>0}) \simeq p^\ast \left[ \det(T \cM|_{\cZ_\xi^\ss}^{<0})^\dual \otimes \det(\cP|_{\cZ_\xi^\ss}^{>0}) \right].
\]
For any point in $\cZ_\xi^{\ss}$, classifying a map $B\bG_m \to \cX$, we can compose with $p$ to get a map $\gamma : B\bG_m \to \bung$ (this corresponds to forgetting the Higgs field and the auxiliary map to $B\bG_m$). So we have
\[
\eta_\xi = \op{wt}_\xi \left( \det(T \bung|_{B\bG_m}^{<0})^\dual \otimes \det(\cP|_{B\bG_m}^{>0}) \right),\]
where we restrict to $B\bG_m$ along $\gamma$.

We let $\fg^{>0}$ and $\fg^{<0}$ denote the subspace with positive (resp. negative) weights with respect to $\xi$, and we compute
\begin{gather*}
\cP|_{B\bG_m}^{>0} \simeq R\Gamma(\Sigma, L^\dual \otimes K \otimes E^\ast(\fg^{>0})[1]), \text{ and} \\
(T\bung|_{B\bG_m}^{<0})^\dual \simeq R\Gamma(\Sigma,E^\ast(\fg^{<0})[1])^\dual \simeq R\Gamma(\Sigma,E^\ast(\fg^{>0}) \otimes K).
\end{gather*}
Using the fact that $L^\dual \otimes K [1] \oplus K \sim \cO_x^{\oplus \deg(L)}$ in $K^0(\Perf(\Sigma))$ we have
\begin{align*}
\eta_\xi &= \op{wt}_\xi \det( R\Gamma(\Sigma, E^\ast(\fg^{>0}) \otimes K \otimes (L^\dual [1] \oplus \cO_\Sigma))) \\
&=  \deg(L) \cdot \op{wt}_\xi \det(\fg^{>0}) = \deg(L) \sum_{\alpha \in \Phi^+} \alpha(\xi) = 2 \deg(L) \rho(\xi).
\end{align*}
Here $\rho$ denotes half the sum of positive roots of $G$. Observe that if $\deg(L) \geq 0$, then $\eta_\xi \geq 0$ for all $\xi$. In particular this happens for the most commonly studied case $L = K$ as long as $g>0$.

\subsection{$C$-weights and connected components}

Using the central action of $C$, we will produce index formulas for each individual component of $\Higgs_{G,L}$.

\begin{lem} \label{lem:component_weights}
If $\cL \in \Pic(\cX_{G,L})$ has level $h$, then for any $\gamma \in \pi_1(G)$ the $C$-weight of $\cL|_{\cX_{\ttype}}$ is $\iota(\gamma) h$.
\end{lem}
\begin{proof}
The claim only depends on $c_1(\cL)$ and is linear in $h$, so it suffices to show this for $D_\Sigma(V)$. Let $V = \bigoplus_{\mu \in \bar{C}} V_\mu$ be the decomposition into $C$-weights. Using the fact that $E_\Sigma^\ast V_\mu$ has $C$-weight $\mu$ and $\det(V_\mu) = \bC_{\op{rk}(V_\mu) \cdot \mu}$ is an integral (as opposed to fractional) character of $G$, one computes that for $E \in \cM_G$, $D_\Sigma(V)_E$ is concentrated in $C$-weight
\[- \sum \deg(E^\ast \bC_{\op{rk}(V_\mu) \cdot \mu}) \cdot \mu = -\sum \op{rk}(V_\mu) \langle \gamma, \mu \rangle \cdot \mu = - \op{Tr}_V(\gamma \cdot (-)),
\]
regarding $\gamma$ as an element of $N_\bQ$.
\end{proof}

Note that when $h$ is positive, any $\gamma$ is determined uniquely by $\iota(\gamma) h$. Therefore when $\cL$ has a positive level it has distinct $C$-weights on the connected components of $\cX$. \autoref{lem:component_weights} implies that $R\Gamma(\cX_{\ttype},\cL)$ will vanish for $\gamma \neq 0$, and more generally
\begin{gather*}
R\Gamma(\cX_{\ttype},\cL \otimes E_x (V)) \simeq R\Gamma(\cX_{\ttype},\cL \otimes E_x(V_\mu)) \text{and} \\
R\Gamma(\cX_{\ttype}^{\ss},\cL \otimes E_x^\ast V) \simeq R\Gamma(\cX_{\ttype}^{\ss},\cL \otimes E_x^\ast V_\mu),
\end{gather*}
where $\mu = - \iota(\gamma) h$ and $V_\mu$ denotes the $\mu$-isotypical summand of $V$.

\begin{ex}
Consider the case of $G=GL_n$, and let $\sigma_1,\dots,\sigma_n$ be a basis for the lattice $N$ of cocharacters (hence also a basis for the Lie algebra of the maximal torus). A positive level is a quadratic form $h = (h_2-h_1)(\sigma_1^2 + \cdots + \sigma_n^2) - h_2 (\sigma_1 + \cdots + \sigma_n)^2$ where $h_1>h_2>0$ are integers. This is the level of the invertible sheaf $\cL = D_\Sigma(\bC^n)^{h_1-h_2} \otimes D_\Sigma(\det)^{h_2}$, where $\bC^n$ is the standard representation and $\det$ is the determinant character. The topological type of a vector bundles is classified by its degree $d$, corresponding to the element $\frac{d}{n} (\sigma_1 + \cdots + \sigma_n) \in \pi_1(G) \subset N_\bQ$. Likewise the character lattice of $G$ is $\bZ \cdot (\sigma_1^\dual + \cdots + \sigma_n^\dual)$, where $\sigma_i^\dual$ is the dual basis for the lattice $M = N^\dual$. The discussion above shows that $\cL(\mu)$ will have global sections on the component of $\Higgs_{G,L,\ttype}$ classifying Higgs bundles of degree $d$ only if
\[
\mu = \iota(\gamma) h = (d h_2 + \frac{d}{n} (h_1-h_2)) (\sigma_1^\dual + \cdots + \sigma_n^\dual)
\]
is integral, i.e. if and only if $d h_1 \equiv d h_2 \mod n$.
\end{ex}

\subsection{When quantization commutes with reduction}

For any locally free sheaf of the form $L \otimes E_x^\ast U$, \autoref{thm:virtual_localization} implies that $\chi(\cX,L \otimes E_x^\ast U) = \chi(\cX^{\ss},L \otimes E_x^\ast U)$ as long as for all $\xi$, the $\xi$-weights of $L \otimes E_x^\ast U|_{\cZ_\alpha^{\ss}}$ are $< \eta_\xi$. The following is an observation of \cite[\S 1.11]{teleman2009index} in the case of the moduli of $G$-bundles, which we prove here for the reader's convenience:

\begin{lem} \label{lem:weights}
Let $\cL \in \Pic(\cX)$ have level $h$. Then $\cL|_{\cZ_\xi^\ss}$ has $\xi$-weight $h(\xi,\xi)$.
\end{lem}

\begin{proof}
Let $\sigma : \cM_{G_\xi} \to \cM_G$ be the map assigning a family of bundles for the Levi subgroup $G_\xi \subset G$ to the induced family of $G$ bundles. It is evident from the definition that $\sigma^\ast : \Pic(\cM_G) \to \Pic(\cM_{G_\xi})$ maps an invertible sheaf of level $h$ to one of level $h|_{\fg_\xi}$. On the other hand, the center of $G_\xi$ contains the one parameter subgroup corresponding to $\xi$, and by definition a point on $\cZ_\xi^{\ss} \subset \cM_{G_\xi}$ classifies a bundle of topological type $\xi \in \pi_1(G_\xi) \subset N_\bQ$. Thus \autoref{lem:component_weights} implies that the $\xi$-weight is $(\iota(\xi) (h|_{\fg_\xi}))(\xi) = h(\xi,\xi)$.
\end{proof}

We can now establish a version of the ``quantization commutes with reduction" theorem.

\begin{prop}\label{prop:quantization}
Let $\cL \in \Pic(\cX_{G,L})$ have a positive level $h$ such that $\mu = -\iota(\gamma)h \in M_\bQ^W$ is integral. Then $R\Gamma(\cX_{\ttype},\cL(\mu)) \to R\Gamma(\cX_{\ttype}^{\ss}, \cL(\mu))$ is an equivalence if
\begin{equation} \label{eqn:level_condition}
h(\xi,\xi) < 2 \deg(L) \rho(\xi), \quad \forall \xi \in N_\bQ^+ \cap N_\bQ', \xi \neq 0.
\end{equation}
In particular this holds if $\deg(L) \geq 0$.
\end{prop}
\begin{proof}
Observe that $E_x^\ast \bC_\mu|_{\cZ_\xi^\ss}$ has $\xi$-weight $\mu(\xi)$ for any $\xi$ and any character $\mu$ of $G$. We combine this with \autoref{lem:weights} and the computation of $\eta_\xi$ above. Then \autoref{prop:quantization} implies the conclusion of the theorem, provided that
\[
h(\xi,\xi) + \mu(\xi) = h(\xi,\xi) - h(\gamma,\xi)< \eta_\xi = 2 \deg(L) \rho(\xi)
\]
for all $\xi \in \gamma + (N_\bQ^+ \cap N_\bQ')$ except for $\xi = \gamma$ (which corresponds to the semistable locus of $\cX_{\ttype}$). The claim of the proposition follows from rewriting this in terms of $\xi'$, where $\xi = \gamma + \xi'$ under the direct sum decomposition $N_\bQ = N_\bQ^+ \oplus N_\bQ'$. Because $h$ is invariant $\xi'$ and $\gamma$ are orthogonal, so $h(\xi,\xi) - h(\gamma,\xi) = h(\xi',\xi')$.
\end{proof}

\begin{rem} \label{rem:positive_enough}
For completeness, we note that the proof above leads to a more general statement: Let $M_\bQ = M_\bQ^W \oplus M_\bQ'$ denote the canonical decomposition into $W$-trivial and $W$-nontrivial summands. Fix a representation $V$ and let $\Omega \subset M_\bQ'$ be the points $\mu'$ for which $-\iota(\gamma) h + \mu'$ is a weight appearing in the character of $V$. Then \autoref{prop:quantization} implies that for any invertible sheaf $\cL$ of level $h$
\[
R\Gamma(\cX_{\ttype},\cL \otimes E_x^\ast(V)) \simeq R\Gamma(\cX_{\ttype}^{\ss},\cL \otimes E_x^\ast(V))
\]
as long as
\begin{equation}\label{eqn:ineq}
h(\xi',\xi') + \mu'(\xi') < 2 \deg(L) \rho(\xi'),
\end{equation}
for all $\mu' \in \Omega$ and all nonzero $\xi' \in N_\bQ^+ \cap N_\bQ'$ for which $\gamma+\xi'$ appears as a label for some Harder-Narasimhan stratum of $\cX_{\ttype}$.

This condition on $h$ is an infinite intersection of open half-spaces, hence defines a convex subset of the positive cone $C_\Omega \subset H^4(BG;\bQ)$. If we fix $\mu' \in \Omega$, then \eqref{eqn:ineq} is an inequality between the negative definite quadratic form $h$ and a linear form $2 \deg(L) \rho - \mu'$ on a discrete subset of $N_\bQ'$ which does not include the origin, so for any positive level $h$ we have $kh \in C_\Omega$ for $k \gg 0$.
\end{rem}

\begin{ex}
If $\cL \in \Pic(\cX)$ has positive level $h$ and $\mu = \iota(\gamma) h$ is integral, then $\chi(\cX_{\ttype},\cL(\mu)^k \otimes E_x^\ast V) = \chi(\cX_{\ttype}^{\ss},\cL(\mu)^k \otimes E_x^\ast V)$ for $k\gg0$, regardless of $\deg(L)$.
\end{ex}

\subsection{The main index formula}

We will unpack the discussion above in slightly more concrete terms. Given an invertible sheaf $\cL \in \Pic(\cX_{G,L})$ with positive level $h$, we will also use $\cL$ to denote its restriction to the stack of Higgs bundles $\Higgs_{G,L}$ via the quotient map $\Higgs_{G,L} \to \cX_{G,L}$. The vector spaces $H^i(\Higgs^{\ss},\cL(\mu))$ have a canonical $\bG_m$ action coming from scaling the Higgs field,\footnote{More precisely, we have a cartesian square, with flat maps, \[\xymatrix@R=10pt{\Higgs^\ss_{G,L,\ttype} \ar[d] \ar[r] & \cX^\ss_{G,L,\ttype} \ar[d]^{\pi} \\ \op{pt} \ar[r] & B\bG_m}.\] The base change formula implies that $R\Gamma(\Higgs^\ss_\ttype,\cL(\mu)) \simeq R\pi_\ast(\cL(\mu))|_{\op{pt}}$, which is equivalent to specifying a grading on the former. The weight $n$ graded summand is canonically isomorphic to $R\Gamma(\cX^\ss_\ttype,\cL(\mu)\langle n \rangle)$.} and we define the graded Euler characteristic
\[
\chi_{\bC^\ast}(\Higgs_{\ttype}^{\ss},\cL(\mu)) = \sum_{n\geq 0} t^n \sum_i (-1)^i \dim (H^i(\Higgs_{\ttype}^{\ss},\cL(\mu))_{\text{weight }n}).
\]
Recall from above that this vanishes unless $\mu = -\iota(\gamma) h$. We recall the notation $F^{reg}_{h,t}$ and $\theta_t(f)$ for \S \ref{sect:index_formula}. Combining \autoref{prop:index_formula} with \autoref{prop:quantization} gives:

\begin{thm}[Equivariant Verlinde index formula] \label{thm:main_index}
Let $\cL \in \Pic(\cX)$ have positive level $h$. Then for any $\gamma \in \pi_1(G)$ for which $\mu = \iota(\gamma) h \in M_\bQ^W$ is integral we have\footnote{As noted above, the equivariant index $\chi_{\bC^\ast}$, and in fact $R\Gamma$, vanishes for $\cL(\mu)$ for any $\mu$ other than $\mu = \iota(\gamma)h$.}
\begin{equation} \label{eqn:main_index}
\chi_{\bC^\ast} \left(\Higgs_{G,L,\ttype}^{\ss}, \cL(\mu) \right)= (1-t)^{\#_L}\sum_{f_t \in F_{h,t}^{reg} / W} \left(\prod_\alpha (1-t e^\alpha(f_t))\right)^{-\chi(L)} \theta_t(f_t)^{1-g} e^{-\mu} (f_t),
\end{equation}
if $\deg(L)\geq 0$, or if $\deg(L)<0$ and $h$ is sufficiently positive (\autoref{rem:positive_enough}). Here $f_t$ ranges over a complete set of Weyl orbit representatives, with $\chi(L) = \deg + 1-g$ and $\#_L := -\rank(\fg) \chi(L) - \dim(\fz) h^1(L)$.
\end{thm}

\begin{rem}
We can rewrite the left-hand side of \eqref{eqn:main_index} a bit more concretely in terms of the underlying classical stack as
\[
\chi_{\bC^\ast}(\Higgs_{\ttype}^{\ss},\cL(\mu)) = \chi_{\bC^\ast}(\Higgs_{\ttype}^{\rm{cl},\ss},\cL(\mu) \otimes \cO_{\Higgs}^{\rm{vir}}),
\]
where $\cO_{\Higgs}^{\rm{vir}} = \bigoplus_i H_i(\cO_{\Higgs})[i]$ is the virtual structure sheaf. See \S \ref{sect:verlinde_final} for a discussion of when $\cH^{\ss}_{G,L,\ttype}$ is actually classical.
\end{rem}

\begin{rem}
If instead of $\Higgs_{G,L}$, one wanted to consider the non-reduced stack of $L$-valued Higgs bundles
\[
\iSpec_{\cM}(\Sym(R\pi_\ast(L^\dual \otimes K \otimes E^\ast(\fg)[1]))),
\]
as discussed in \autoref{rem:reducedness}, then all of the previous computations work with minimal modification. The only difference is that in the statement of \autoref{prop:index_formula} and \autoref{thm:main_index} one must take $\#_L = \rank(\fg) (g-1- \deg(L))$.
\end{rem}

\begin{ex} \label{ex:gukov_pei}
When $L = (\sqrt{K})^R$ for $R \in \bZ$, the index formula of \autoref{prop:index_formula} mostly agrees with the ``equivariant Verlinde formula" proposed in \cite{gukov2015equivariant}, which computes the partition function of the ``$\beta$-deformed complex Chern-Simons theory on $\Sigma \times S^1$," where $R$ denotes the assignment of $R$-charge for the adjoint chiral multiplet. Gukov and Pei's formula uses
\[
\#_L = \rank(\fg) (g-1-\deg(L)) = \rank(\fg) (1-R)(g-1),
\]
as in the previous remark. These agree when $\deg(L)>2g-2$, but in the case where $L=K$, our choice is the one which can be interpreted as the graded dimension of the space of global sections on some classical stack or scheme (See \autoref{thm:main_verlinde}). When $\deg(L)<2g-2$, e.g. when the $R$-charge is $<2$, the stack $\Higgs_{G,L,\ttype}$ may have non-trivial derived structure, and a priori this can contribute $H^i$ for $i<0$ to the index \eqref{eqn:main_index}.
\end{ex}

\subsection{The equivariant Verlinde formula}
\label{sect:verlinde_final}

In this section we offer a more recognizable and geometric version of the Verlinde formula. Adopting the notation of \autoref{thm:main_index}, we have

\begin{thm}[Equivariant Verlinde formula] \label{thm:main_verlinde}
If either
\begin{enumerate}
\item $\deg(L) > \max(0,2g-2)$, or
\item $L=K$ and $g>1$,
\end{enumerate}
then $\Higgs_{G,L,\gamma}^\ss$ is classical and $H^i(R\Gamma(\Higgs_{G,L,\ttype}^\ss,\cL(\mu)))$ vanishes for $i\neq0$. Thus the right hand side of \autoref{eqn:main_index} computes the graded dimension \[\dim_{\bC^\ast} H^0(\Higgs^{\rm{cl},\ss}_{G,L,\ttype},\cL(\mu)).\]
\end{thm}
There are two proofs. The first, which is simpler, applies only when $L=K$ and for certain connected components of $\Higgs_{G,K}^\ss$ which we show are smooth (they are gerbes over a smooth projective-over-affine DM stack). The general claim relies on a vanishing theorem proved in the appendix, which builds on earlier vanishing theorems of \cite{frenkel2013geometric, fishel2008strong, teleman1998borel}.

\subsubsection{The smooth case}

Let us call a class $\xi \in \pi_1(G)$ \emph{primitive} if for all nontrivial $\xi \in N^+$, $\gamma \in \pi_1(G)_\bQ \subset \pi_1(G_\xi)_\bQ \subset N_\bQ$ does not lie in $\pi_1(G_\xi)$.
\begin{ex}
When $G = \GL_n$, a class in $\pi_1(G)$ is primitive if and only if the rank and degree of the corresponding vector bundle are coprime.
\end{ex}
\begin{lem} \label{lem:smoothness}
If $\gamma \in \pi_1(G)$ is primitive, then $\Higgs_{G,K,\ttype}^{\ss}$ is smooth.
\end{lem}
\begin{proof}
By \autoref{rem:cotangent_stack} $\Higgs_{G,K}^{\ss}$ is a $C$-gerbe over $T^\ast(\cM_G \thickslash C)^\ss$, so the former is smooth if and only if the latter is. $T^\ast(\cM_G \thickslash C)^\ss$ is quasi-smooth, and its cotangent complex is self dual. That implies that it is smooth if and only if
\[
0 = \dim H_1(L_{T^\ast(\cM_G \thickslash C),(E,\phi)}) = \dim \op{Aut}_{\cM_G \thickslash C}(E,\phi),
\]
for every semistable point $(E,\phi)$ in $T^\ast(\cM/G)$. Because $g>1$ the generic stabilizer of a point in $\cM^{\ss}_G$ is $C$, so this computation shows that if $\Higgs_{G,K,\ttype}^{\ss}$ fails to be smooth, then it contains a strictly semistable point.

If there is a strictly semistable point, one can consider the associated graded of the Jordan-H\"older filtration to obtain a polystable point\cite{otero2010jordan}. In other words there is a dominant cocharacter $\xi \in N^+$ and a semistable $G_\xi$-Higgs bundle $(E,\phi)$ whose induction $(E_G,\phi_G)$ to $G$ is semistable of topological type $\gamma$. Let $\gamma' \in \pi_1(G_\xi) \subset N_\bQ$ classify the topological type of $E$, then we claim that $\gamma' = \gamma$, and hence $\gamma$ in not primitive: Indeed $\gamma'$ maps to $\gamma$ under the projection $\pi_1(G_\xi)_\bQ \to \pi_1(G)_\bQ$, but we claim that $\gamma' \in \pi_1(G)_\bQ \subset \pi_1(G_\xi)_\bQ$ already. If not, then some power of $\gamma'$ would correspond to a central cocharacter of $G_\xi$ which, together with the tautological reduction of structure group of $(E_G,\phi_G)$ to the standard parabolic $P_{\gamma'}$, would destabilize $(E_G,\phi_G)$.
\end{proof}

As smooth derived stacks are classical, this implies that $\Higgs_{G,K,\ttype}^\ss = \Higgs_{G,K,\ttype}^{\ss,\rm{cl}}$ for primitive $\gamma$. Note that this lemma is why we have used the reduced stack of Higgs bundles instead of $T^\ast \cM_G$ -- the latter is not smooth at any point unless $Z(G)$ is finite.

\begin{proof}[Proof of \autoref{thm:main_verlinde} when $K=L$ and primitive $\gamma$]
$\Higgs_{G,K,\ttype}$ is classical and smooth by \autoref{lem:smoothness}. In fact, the proof established that $\Higgs^{\ss}_{G,K,\ttype} \thickslash C \simeq T^\ast(\cM_G\thickslash C)^\ss$ is a smooth projective-over-affine Deligne-Mumford stack with trivial canonical bundle. $C$ acts trivially on the bundle $\cL(\mu)$, so it descends to $\Higgs^{\ss}_{G,K,\ttype} \thickslash C$. The description of line bundles on $\cM_G$ which restrict to ample bundles on $\cM_G^\ss$ of \cite{teleman2000quantization} and the fact that they also induce ample bundles on $\Higgs_{G,K}^\ss$ (see the GIT construction in \cite{simpson1994modulib}), shows that $\cL(\mu)$ is ample.\footnote{By ample, we mean that it descends, after a suitable power, to an ample bundle on the good moduli space $\Higgs_{G,K,\ttype}^{\ss} \thickslash C \to H_{G,K,\ttype}$, see \cite{kresch2009geometry}.} The cohomology vanishing of the claim follows from Grauert–-Riemenschneider for projective-over-affine Deligne-Mumford stacks, along with the fact that the pushforward functor $\QCoh(\Higgs^{\ss}_{G,K,\ttype}) \to \QCoh(\Higgs^{\ss}_{G,K,\ttype} \thickslash C)$ is exact (as $C$ is reductive).
\end{proof}

\subsubsection{Proof of \autoref{thm:main_verlinde} in the general case}

For general $L$ or general $\gamma$, the stack $\Higgs^{\ss}_{G,L,\ttype}$ will not be smooth, so we cannot simply apply Grauert--Riemenschneider. However, if $\Higgs_{G,L,\ttype}$ is classical (i.e. $H_i(\cO_X) = 0$ for $i>0$) then $H^i(\Higgs_{G,L,\ttype}^\ss,\cL(\mu))$ automatically vanishes in negative degrees. To prove vanishing for $i>0$, \autoref{prop:quantization} implies that $H^i(\Higgs_{G,L,\ttype}^\ss,\cL(\mu)) \simeq H^i(\Higgs_{G,L,\ttype},\cL(\mu))$, which vanishes by \autoref{thm:vanish}. So \autoref{thm:main_verlinde} amounts to specifying conditions under which $\Higgs_{G,L,\ttype} = \Higgs^{\op{cl}}_{G,L,\ttype}$, which happens if and only if $\dim(\Higgs_{\ttype}^{\rm{cl},\ss}) = \op{vdim}(\Higgs^{\ss}_{\ttype}) = \dim(\fg) \deg(L) + \dim(\fz) h^1(L)$.

\begin{prop} \label{prop:classical}
$\Higgs_{G,L}$ is a classical stack in the following cases:
\begin{enumerate}
\item if $\deg(L) > \max(0,2g-2)$ \cite{arinkin2016irreducible};
\item if $L = K$ and $g>1$;
\end{enumerate}
\end{prop}
\begin{proof}
(1) is part of the statement of \cite[Proposition 3.2]{arinkin2016irreducible}. (2) has been known for some time for semisimple $G$ -- see \cite[Proposition 2.1.2]{beilinson1991quantization} and the references therein. As remarked there, the semisimple case implies that $\Higgs_{G,K}$ is classical for general reductive $G$: The projection $G \to G^{\rm{ad}} = G/Z(G)$ induces a map of classical stacks
\[
\iSpec_{\cM_G} (\Sym(\cO_{\cM_G} \otimes H^1(\Sigma,\cO_\Sigma)\otimes \fz)) \to \cM_{G^{\rm{ad}}}
\]
which is smooth of relative dimension $(2 g-1) \dim(\fz)$. By definition $\Higgs_{G,K}^{\op{cl}}$ is the pullback of $\Higgs_{G^{\rm{ad}},K}^{\op{cl}}$ along this map, so $\dim \Higgs_{G,K}^{\op{cl}} = (2g-2) \dim(\fg) + \dim(\fz)$ as needed.
\end{proof}

\begin{rem}
Note that $\cM_G \subset \Higgs_{G,L,\ttype}$ is a closed substack, so $\dim \Higgs^{\ss}_\ttype = \op{vdim}\Higgs_{\ttype}^\ss$ can only happen if $\dim \cM_G = (g-1) \dim(\fg) \leq \dim(\fg) \deg(L) + \dim(\fz) h^1(L)$.\footnote{This was pointed out to us by Dima Arinkin.} We do not know of a complete classification of when $\Higgs_{G,L,\ttype}^\ss$ is classical for $L$ with
\[
g-1 - \frac{\dim \fz}{\dim \fg} h^1(L) \leq \deg(L) \leq 2g-2.
\]
\end{rem}

\subsection{Comparison with the good moduli space}
\label{sect:good_moduli}

As mentioned above, we have a projective-over-affine good moduli space which we denote $q : \Higgs_{G,L,\ttype}^{\rm{cl},\ss} \to H_{G,L,\ttype}$. We record the following lemma for classical algebraic stacks
\begin{lem}
Let $\cY \to \cY'$ be a $A$-bundle for some algebraic group $A$, and let $\cY \to Y$ be a good moduli space, then $A$ acts on $Y$, and we have a Cartesian diagram
\[
\xymatrix@R=10pt{\cY \ar[r] \ar[d] & \cY' \ar[d] \\ Y \ar[r] & Y/A}
\]
\end{lem}
\begin{proof}
The formation of good moduli spaces is functorial and commutes with products with algebraic spaces. Letting $\cY_\bullet$ be the Cech nerve of the map $\cY \to \cY'$, one has $\cY_n \simeq A^n \times \cY$. Taking the good moduli space at every level, one arrives at a simplicial space $Y_\bullet$ with $Y_n \simeq A^n \times Y$. Under this identification is will be the simplicial nerve of a groupoid $A \times Y \rightrightarrows Y$ corresponding to an $A$-action on $Y$. This map of simplicial stacks, for which all face maps are Cartesian, is the presentation for the desired map $\cY' \to Y/A$.
\end{proof}

\begin{rem} \label{rem:properness}
If $A$ is linearly reductive, $\cY'$ is a quotient stack, and $Y$ is a projective-over-affine variety such that $\Gamma(Y,\cO_Y)^A$ is finite dimensional, it follows that $\cY'$ is cohomologically proper in the sense of \cite{halpern2014mapping}. In particular all perfect complexes have finite dimensional derived global sections. This applies, in particular, to the stack $\cX_{G,L,\ttype}$.
\end{rem}
In our case $A = \bG_m$, and we denote the map $\cX^\ss \to H/\bG_m$ by $q'$. For $F \in \QCoh(\cX^\ss)$, the base change theorem gives us a canonical equivalence $q_\ast(F|_{\Higgs^\ss}) \simeq (q')_\ast(F)|_{H}$. Hence $q_\ast(F|_{\Higgs^\ss})$ has a canonical $\bG_m$-action. As the pushforward functors $q_\ast$ and hence $(q')_\ast$ are exact, we have $R\Gamma(\Higgs_{G,L,\ttype}^{\rm{cl},\ss},\cL(\mu)) \simeq R\Gamma(H_{G,L,\ttype},q_\ast(\cL(\mu)))$ as $\bG_m$-modules. Hence we have 
\begin{cor}
\eqref{eqn:main_index} computes the dimension of $H^0(H_{G,L,\ttype},q_\ast(\cL(\mu)))$ under the hypotheses of \autoref{prop:classical}.
\end{cor}
Note also that as $\cL(\mu)$ has trivial $C$-weight, some power $\cL(\mu)^k$ for $k>0$ descends to an invertible sheaf on $H_{G,L,\ttype}$, which will agree with $q_\ast(\cL(\mu)^k)$.

\subsection{Higgs bundles framed at a point}
\label{sect:framed}

We can re-interpret the formulas above in terms of the following:

\begin{defn}
Given a $G$-scheme $F$, we define the moduli stack $\Higgs^{F\!\mbox{-}\rm{fr},\ss}_{G,L}$ of \emph{$L$-valued semistable $G$-Higgs bundles with $F$-framing at $x \in \Sigma$} to be the total space of the $F$-fiber bundle associated to the principal $G$-bundle
\[
\Higgs^\ss_{G,L} \to \cM_{G,L} \xrightarrow{\op{ev}_x} BG.
\]
\end{defn}

Note that the ``semistable" superscript does not necessarily mean that these are semistable points of the analogous stack of all $F$-framed $G$-Higgs bundles.

\begin{ex}
When $F = G$ with left multiplication, then this is the usual notion of semistable Higgs bundles with framing (i.e. trivialization) at the point $x$. The non-abelian Hodge correspondence implies that the component of $\Higgs^{G\!\mbox{-}\rm{fr},\ss}_{G,K}$ corresponding to topologically trivial $G$-bundles is actually a (classical) scheme.
\end{ex}

\begin{ex}
When $F = G^{\rm{ab}}$ with $G$ acting by left multiplication, then $\Higgs^{F\!\mbox{-}\rm{fr},\ss}$ is the total space of the tautological $G^{\rm{ab}}$ torsor associated to the tautological principal $G$-bundle on $\Higgs^\ss$. When $G=\GL_n$, the stack $\Higgs^{F\!\mbox{-}\rm{fr},\ss}$ consists of Higgs vector bundles $(\cE,\phi : \cE \to \cE \otimes L)$ along with a non-zero element of $\det(\cE)_x$.
\end{ex}

The composition of the map $\Higgs^{G^{\rm{ab}}\!\mbox{-}\rm{fr},\ss} \to \Higgs^\ss$ composed with the map to the $C$-rigidification $\Higgs^\ss \to \Higgs^\ss \thickslash C$ (see the proof of \autoref{thm:main_verlinde} in the smooth case) is a relative gerbe banded by the finite group $\ker(C \to G^{\rm{ab}})$. Therefore $\Higgs^{G^{\rm{ab}}\!\mbox{-}\rm{fr},\ss}$ has the same good moduli space as $\Higgs^{\ss}$. For $\cL \in \Pic(\cX)$, we abuse notation and let $\cL$ denote the restriction to $\Higgs^{G^{\rm{ab}}\!\mbox{-}\rm{fr},\ss}$ as well. In the notation of \autoref{thm:main_index} we have

\begin{cor}
If $\iota(\gamma)h$ is integral then the formula \eqref{eqn:main_index} expresses
\[
\left\{ \begin{array}{ll} \dim_{\bC^\ast} H^0(\Higgs^{G^{\rm{ab}}\!\mbox{-}\rm{fr},\ss}_{G,L,\ttype},\cL), & \text{ under the hypotheses of \autoref{prop:classical}}, \text{ and} \\ \chi_{\bC^\ast}(\Higgs^{G^{\rm{ab}}\!\mbox{-}\rm{fr},\ss}_{G,L,\ttype},\cL), & \text{ if } \deg(L) \geq 0. \end{array} \right.
\]
$R\Gamma(\Higgs^{G^{\rm{ab}}\!\mbox{-}\rm{fr},\ss},\cL)$ vanishes if $\iota(\gamma)h$ is not integral.
\end{cor}

Combining the several of the previous results, \autoref{prop:index_formula} and \autoref{prop:quantization}, with our discussion of when $\Higgs_{G,L}$ is classical, we can state a slightly more general, if slightly weaker, corollary. If $\cL \in \Pic(\cX)$ has positive level $h$ such that $\mu = \iota(\gamma) h$ is integral, we can restrict $\cL$ to $\Higgs^{G\!\mbox{-}\rm{fr},\ss}_{G,L,\ttype}$ and define the canonically $G$-equivariant graded ring
\[
R_\cL = \bigoplus R_{\cL,k} := \bigoplus_k H^0(\Higgs^{G\!\mbox{-}\rm{fr},\ss}_{G,L,\ttype},\cL^k).
\]
As a representation of $G$, $R_{\cL,k}$ is concentrated in $C$-weight $k \mu$.
\begin{cor}
Assume the hypotheses of \autoref{prop:classical}. Then for any irreducible $G$-representation $U$, the formula \eqref{eqn:index} of \autoref{prop:index_formula} gives the graded dimension of the $U\otimes {\bC_{k\mu}}$-isotypical summand of $R_{\cL,k}$ for $k\gg0$.
\end{cor}

\appendix

\section{Vanishing theorem on the stack of Higgs bundles, by Constantin Teleman}																				
\label{appendix:vanishing}

This section proves the following cohomology 
vanishing theorem, which ensures that the $t$-Euler characteristic computed in the body of the paper 
gives the graded dimensions of the space of regular sections of positive line bundles 
over the Higgs moduli stack $\cH_{G,K}$ in genus $\ge2$. For components of 
$\cH_{G,K}$ corresponding to primitive elements of $\pi_1(G)$, the Grauert-Riemenschneider theorem may 
be used, as explained earlier in the paper. For the general case, the rather involved proof below is based 
on results in \cite{fishel2008strong}.
\begin{thm}\label{thm:vanish}
Let $\cL$ be a positive-level line bundle on $\cH_{G,L}$. Then, $H^{>0}(\cH_{G,L}; \cL) =0$.
\end{thm}

\begin{rem}
The `underived' structure of $\cH_{G,K}$ 
ensures the vanishing of negative cohomology. However, there \emph{is} negative-degree cohomology 
in genera $0$ and $1$. (This is easy to see when $g=0$: the semi-stable part of the stack is $T^*BG$, 
with cotangent fibers of degree $1$, and $\cL$ vanishes there, so we are finding the $G$-invariants 
in ${\textstyle\bigwedge}^*\frg$, with $\frg$ placed in degree $(-1)$.) 
\end{rem}

\subsection{Reduction to semi-simple $G$.}  
The vanishing theorem can be reduced to the case when $G$ is semi-simple --- \emph{which we shall 
henceforth assume} --- provided that we include the variants of $\cM_G$ twisted by central elements 
$c\in Z(G)$. (The twist forces a minor tweak in the argument.) Indeed, the derived sequence $G'=[G,G] \to 
G \to G^{ab}$, which splits over a finite cover, leads to a fiber bundle structure of $\cM_G$ over 
$\cM_{G^{ab}}$ with finite structure group, having as fibers the versions of $\cM_{G'}$ twisted by 
elements $c\in Z(G')$.\footnote{Over a component of $\cM_{G^{ab}}$ labelled by $p\in \pi_1G^{ab}$, the  
component $c$ is the projection to $\pi_1G^{ad}/\pi_1G' \cong Z(G')$ of any lift of $p$ to 
$\pi_1G$.} Cohomology vanishing on the base comes from positivity and Kodaira's theorem, so this 
bundle structure reduces us to the semi-simple case.

Before proceeding with the proof, we note the following
\begin{lem}\label{symr}
The cohomology of a positive-level $\cL$ over $\cH_{G,L}$ is determined degree-by-degree along the fibers
of the projection $\cH_{G,L} \to \cM_G$:
\[
H^*(\cH_{G,L}; \cL) = \bigoplus\nolimits_r \bH^*\left(\cM_G; \cL\otimes\mathrm{Sym}^r \cP_L\right).
\]
\end{lem}
\begin{proof}
In question is the commutation of cohomology with the direct sum, and the argument follows  
\cite{frenkel2013geometric}, Lemma~4.13: the cohomology is computed correctly over a finite-type part 
of the stack $\cM_G$, namely a finite union of Atiyah-Bott strata. (The line bundle $\cL$ gives the 
additional positivity needed to remove the restriction $g\ge 2$ which was needed for $\cL=\cO$ in \cite 
{frenkel2013geometric}.)
\end{proof}

We will establish a stronger result, which is local over (symmetric powers of) the curve $\Sigma$. 
Fix $r$ in Lemma~\ref{symr}. The invariant part $\sigma_*^+$ of the  direct image along the 
Galois cover $\sigma:\Sigma^r\to \mathrm{Sym}^r\Sigma$ defines a sheaf  $\cS_r:= \sigma_*^+\left[(L^\vee\otimes 
K \otimes E^*(\frg)[1])^{\boxtimes r}\right]$ over $\cM_G\times\mathrm{Sym}^r\Sigma$, offering an 
identification of sheaves over $\cM_G$
\[
\mathrm{Sym}^r\cP_L := \mathrm{Sym}^r R\pi_*(L^\vee\otimes K \otimes E^*(\frg)[1]) = R\pi_*^s\left(\cS_r\right), 
\]
where $\pi^s$ is the projection to $\cM_G$ along $\mathrm{Sym}^r\Sigma$. 
Integrating $\cS_r$ along $m: \cM_G\times \mathrm{Sym}^r\Sigma \to \mathrm{Sym}^r\Sigma$ instead gives a 
complex of sheaves $\cT_r:=Rm_*\left(\cL\otimes\cS_r \right)$. Switching the order of integrations,
\begin{equation}\label{local}
\bH^*\left(\cM_G; \cL\otimes\mathrm{Sym}^r \cP_L\right) = \bH^*\left(\cM_G;\cL\otimes 
		R\pi^s_*(\cS_r)\right) = \bH^*\left(\mathrm{Sym}^r\Sigma; \cT_r\right).
\end{equation}

The factorisation and cohomology vanishing results of \cite{teleman1998borel} imply that
\begin{enumerate}\itemsep0ex
\item $\cT_r$ is equivalent to a bounded complex of coherent sheaves over $\mathrm{Sym}^r\Sigma$: the 
lower degree is the downshift $(-r)$, while the upper bound depends on the level of $\cL$, and is 
\emph{a priori} linear in $r$. We will see that an upper bound is actually $(-1)$. 
\item The cohomology sheaves of $\cT_r$ have locally constant normal structure along the multi-diagonals 
in $\mathrm{Sym}^r\Sigma$. More precisely, near a point $\{p_1^{r_1}, \dots, p_k^{r_k}\}\in \mathrm{Sym}^r\Sigma$, 
with $p_i\in \Sigma$ pairwise distinct, we can, after choosing formal coordinates 
$z_i$ near the $p_i$, lift the independent translations in the $z_i$ to $\cT_r$, by the 
Knizhnik-Zamolodchikov-Hitchin connection.
\end{enumerate}
\noindent
Recall that the \emph{Serre dual} $\bS\cF$ of a coherent complex $\cF$ on a smooth variety $X$ is 
the complex $R\cH{om}_X\left(\cF;\cK[\dim X]\right)$, where $\cK$ is the dualising sheaf. Serre duality 
asserts that $\bH^q(X; \cF)^\vee = \bH^{-q}(X; \bS\cF)$ if $X$ is proper, otherwise the second cohomology 
must be taken with proper supports. The structure of $\cT_r$ seems difficult to spell out, but 
remarkably, its dual is much cleaner.

\begin{thm}\label{thm:sheaf}
The Serre dual $\bS\cT_r$ of $\cT_r$ is represented by a single coherent sheaf, in degree $0$.
\end{thm}

\noindent
Since $\bH^q\left(\mathrm{Sym}^r\Sigma; \cT_r\right) = \bH^{-q}\left(\mathrm{Sym}^r\Sigma; \bS\cT_r\right)^\vee$ 
vanishes for positive $q$, \autoref{thm:vanish} follows: cf.~\eqref{local}.

\begin{cor}\label{degrees}
The degrees of the cohomology sheaves of $\cT_r$ range from $(-r)$ to at most $(-1)$.
\end{cor}
\begin{proof}
We have $\cT_r\cong \bS\bS\cT_r$, and we can compute the Serre dual of the sheaf $\bS\cT_r$ by using the 
Gysin sequence of the multi-diagonal stratification of $\mathrm{Sym}^r\Sigma$. The highest co-dimension is 
$(r-1)$. Because of the translation-invariance along the strata, the degrees in 
the $R\cH om(\_; \cK)$ on the $E_1$ page range from $0$ to $(r-1)$, shifting down to the claimed range. 
\end{proof}

\subsection{Preliminaries for the proof.} \label{prelim}
Given a point $p\in \Sigma$ with local coordinate $z$, we may present $\cM_G$ as 
the double-coset stack $G[\![z]\!]\backslash \bfX$, for the formal Taylor loop group $G[\![z]\!]$ 
and the \emph{thick flag variety}  $\bfX := G(\!(z)\!)/G[\Sigma^\times]$ (discussed e.g.\ in \cite
{fishel2008strong}, \S7), the quotient of the formal Laurent loop group  $G(\!(z)\!)$ by the (algebraic ind-) 
group $G[\Sigma^\times]$ of algebraic maps from the punctured curve $\Sigma^\times:= \Sigma\setminus \{p\}$ to $G$. 
For the version of $\cM_G$ twisted by $c\in Z(G)$, we need to twist the action of $G[\![z]\!]$ by a fractional 
formal loop $z^{\log c/2\pi i}$. A similar construction works for any finite set of points on $\Sigma$.

There are only finitely many isomorphism classes of irreducible integrable highest-weight representations 
$R$ of $G(\!(z)\!)$ at level opposite to that of $\cL$. Given such an $R$, let $R_0$ be its highest-energy space: 
it is an irreducible representation of $G$ and (when $G$ is connected) determines $R$.

\begin{proof}[Proof of \autoref{thm:sheaf}.]
The formally completed stalk of $\bS\cT_r$ at a general point $\mathbf{p} \in 
\mathrm{Sym}^r\Sigma$ is the vector space dual of the hyper-cohomology $\bH^*_{\mathbf{p}}\left
(\mathrm{Sym}^r\Sigma; \bS\cT_r\right)$ with support at $\mathbf{p}$. We will show that this hyper-cohomology 
is concentrated in degree zero for a point $\mathbf{p} =\{p^r\}$ on the small diagonal. For a general point 
$\mathbf{p} =\{p_1^{r_1}, \dots,  p_k^{r_k}\}$, with pairwise distinct $p_i$, the factorisation in 
\cite{teleman1998borel} expresses the answer as the fusion of the answers at the $p_i$, and the 
asserted purity will therefore hold at any $\mathbf{p}$.

Let us compute the cohomology with supports \emph{before} applying $Rm_*$: more precisely, we seek the 
local cohomology sheaf $\cH_{\cM_G\times \{p^r\}}(\cS_r)$ of $\cS_r$ with supports in 
$\cM_G\times\{p^r\}\subset \cM_G\times\mathrm{Sym}^r\Sigma$. We can project the formal neighborhood back 
to $\cM_G$ along $\pi^s$. Serre duality along this projection shows that cohomology with supports 
replaces $\pi^s_*(\cS_r)$ with 
\[
\mathrm{Sym}^r\left(\hat{L}_p\otimes_{\bC[\![z]\!]} \frg(\!(z)\!)/\frg[\![z]\!]\right) \cong 
\mathrm{Sym}^r\left(\frg(\!(z)\!)/\frg[\![z]\!]\right), 
\]
made into a bundle over $\cM_G$ via the adjoint action of $G[\![z]\!]$ in the double-coset presentation.\footnote 
{The easiest way to see the duality is to pass to $\Sigma^r$ and take symmetric group invariants.} 
That is a colimit of finite-dimensional representations, and their cohomologies are computed correctly on a finite-type 
part of the stack $\cM_G$; cohomology therefore commutes with the colimit, and for the finite bundles we can 
apply the factorisation methods of \cite{teleman1998borel} to obtain 
\[
\begin{split}
\bH^*\left(\cM_G\times\Sym^r\Sigma; \cL\otimes \cH_{\cM_G\times\{p^r\}}(\cS_r)\right) \cong \\
	\bigoplus\nolimits_RH^0&\left(\cM_G; E^*_p(R_0^\vee)\otimes\cL\right)\otimes H^*_{G[\![z]\!]}\left(R\otimes R\Gamma_p(\cS_r)\right),
\end{split}
\]
having used the higher cohomology vanishing for the first factor on the right side \cite{teleman1998borel}.
The proof is concluded by appealing to the result below.

A small amendment is needed for the $c$-twisted version of $\cM_G$: the center $Z(G)$ acts on the set 
of representations $R$, through the outer automorphism $z^{\log c/2\pi i}$, and so $R$ in the factorisation 
above must be replaced by its transform $c(R)$.
\end{proof}

\begin{thm}[\cite{fishel2008strong}, Theorem~E]\label{bigmac}
For any highest-weight representation $R$ of $G(\!(z)\!)$ of strictly negative level, the following 
algebraic group cohomology vanishes in positive degrees:
\[
H^{>0}_{G[\![z]\!]}\left(R\otimes \mathrm{Sym}(\frg(\!(z)\!)/\frg[\![z]\!])\right) = 0. \qed
\]
\end{thm}

We will now give a second proof of \autoref{thm:vanish}, which tracks part of the computation in 
$H^*(\cH_{G,L};\cO)$ from \cite{frenkel2013geometric}, Theorem~4.2, because it passes through a 
vanishing result of independent interest, Proposition~\ref{subdiagonal} below. 
To study $H^*\left(\cM_G; \mathrm{Sym}\,T\otimes\cL\right)$, we build a vector bundle over $\bfX$ from $\frg[\Sigma^\times] := \Gamma(\Sigma^\times;\frg)$ using the adjoint action.\footnote{Use $\Gamma(\Sigma^\times;\frg\otimes L^\vee\otimes K)$ instead,  for general $L$.} Recall the 
following results.

\begin{lem} [using \cite{teleman1998borel} \S8]\label{highestweight} {\ }
\begin{enumerate}\itemsep0ex
\item For all $q$, $H^q\left(\bfX; \cL\otimes{\textstyle\bigwedge}^s\frg[\Sigma^\times]\right)$ is a sum of 
integrable highest-weight representations of $G(\!(z)\!)$ of level opposite\footnote{Our parametrisation of the 
formal disk at $p$ flips the sign of the level and energy.} to that of $\cL$. 
\item The co-factor of $R$ as in \S\S\ref{prelim} is 
$H^q\left(\cM_G; E^*_p(R_0^\vee)\otimes\cL\otimes{\textstyle\bigwedge}^s\frg[\Sigma^\times]\right)$. 
 \qed
\end{enumerate}
\end{lem}

\autoref{thm:sheaf}  lets us say more:

\begin{prop}\label{subdiagonal}
The cohomologies in Lemma~\ref{highestweight} vanish if $q\ge s$, save when $q=s=0$.
\end{prop}

\begin{proof}
We use the same method of pushing forward $E^*(\frg)$ to the symmetric power of $\Sigma$ and integrating 
over $\cM_G$ first; this produces the complex $\cT_s[-s]$, in degrees ranging from $0$ to $s-1$. Because 
$\Sigma^\times$ is affine, no further cohomology appears.  
\end{proof}

\begin{proof}[Second proof of \autoref{thm:vanish}]
To keep notation simple, assume that $L=K$; the argument is not changed by passing to a general $L$. 
Resolve the tangent complex $T=\cP_K$ of $\cM_G$ as
\[
\frg[\Sigma^\times] \xrightarrow{\partial=\mathrm{Ad}_\phi} \frg(\!(z)\!)/\frg[\![z]\!]
\]
with differential, at the point $\phi\in G(\!(z)\!)$, equal to the adjoint twist of the obvious inclusion. 
We seek the vanishing of positive degree hypercohomologies
\begin{equation}\label{hypercoh}
H^q\left(\cM_G;\mathrm{Sym}^r\, T \otimes\cL \right) = \bH^q_{G[\![z]\!]}\left(\mathbf{X};
    \cL\otimes\left\{\bigoplus_{s+t=r}{\textstyle\bigwedge}^s\frg[\Sigma^\times] \otimes \mathrm{Sym}^t(\frg(\!(z)\!)
    /\frg[\![z]\!]),\:\partial\right\} \right),
\end{equation}
with the generating bundles $\frg[\Sigma^\times]$ and $\frg(\!(z)\!)/\frg[\![z]\!]$ placed in cohomological degrees 
$-1$, respectively $0$, and differential induced from $\partial$ above. The $\mathrm{Sym}$-factors carry the 
adjoint action of $G[\![z]\!]$, the variety $\bfX$ carries the natural translation action. The (finite length) 
filtration by $s$-degree gives a convergent spectral sequence $E_1^{-s,q}\Rightarrow \bH^{q-s}$ with first page 
obtained by fixing $s$ and ignoring the differential $\partial$ above. Lemma~\ref{highestweight} gives the 
following factorisation, with $R$ ranging over irreducible integrable loop group representations: 
\[
E_1^{-s,q}\cong \bigoplus_{u;R} H^{q-u}_{G[\![z]\!]}\left(R\otimes \mathrm{Sym}^{r-s}(\frg(\!(z)\!)/\frg[\![z]\!])\right) 
	\otimes H^u\left(\cM_G; E_p^*(R_0^\vee)\otimes\cL\otimes{\textstyle\bigwedge}^s\frg[\Sigma^\times]\right)
\]
Now, Theorem~\ref{bigmac} tells us that $u=q$, and Corollary~\ref{degrees} ensures that $q=u\le s$. 
So the $E_1$ page contains no positive-degree cohomology.
\end{proof}

\bibliography{verlinde_references}
\bibliographystyle{plain}

\end{document}